\documentclass[reqno]{amsart}

\numberwithin{equation}{section}
\usepackage{latexsym}
\usepackage{amsmath}
\usepackage{mathtools}
\usepackage{amssymb}
\usepackage{mathrsfs}
\usepackage{graphicx,colordvi}
\usepackage{upgreek}
\usepackage{ifthen}
\usepackage{bm}
\usepackage{color}

\usepackage[T1]{fontenc}
\usepackage[latin1]{inputenc}

\numberwithin{equation}{section}

\newtheorem{defi}{Definition}[section]
\newtheorem{theorem}[defi]{Theorem}

\newtheorem{corollary}[defi]{Corollary}
\newtheorem{proposition}[defi]{Proposition}
\newtheorem{remark}[defi]{Remark}
\newtheorem{remarks}[defi]{Remarks}

\usepackage{latexsym}
\usepackage{amsmath}
\usepackage{amssymb}

\newcommand{\cF}{{\mathcal F}}

\newcommand{\cH}{{\mathcal H}}
\newcommand{\cB}{{\mathcal B}}

\newcommand{\CC}{{\mathbb C}}

\newcommand{\BB}{{\mathbb B}}
\newcommand{\EE}{{\mathbb E}}

\newcommand{\NN}{{\mathbb N}}
\newcommand{\R}{{\mathbb R}}
\newcommand{\RR}{{\mathbb R}}

\newcommand{\PP}{{\mathbb P}}

\renewcommand{\epsilon}{\varepsilon}

\begin{document}

\title[Critical Spaces]{Critical Spaces For Quasilinear Parabolic Evolution Equations and Applications}

\author{Jan Pr\"uss}
\address{Martin-Luther-Universit\"at Halle-Witten\-berg\\
         Institut f\"ur Mathematik \\
         Theodor-Lieser-Strasse 5\\
         D-06120 Halle, Germany}
\email{jan.pruess@mathematik.uni-halle.de}

\author{Gieri Simonett}
\address{Department of Mathematics\\
        Vanderbilt University\\
        Nashville, Tennessee\\
        USA}
\email{gieri.simonett@vanderbilt.edu}

\author{Mathias Wilke}
\address{Universit\"at Regensburg\\
Fakult\"at f\"ur Mathematik\\
D-93040 Regensburg, Germany}
\email{mathias.wilke@ur.de}

\thanks{This work was supported by a grant from the Simons Foundation (\#426729, Gieri Simonett).}

\subjclass[2010]{Primary: 
35K58,  
35K59,  
35K90,  
35B40,  
Secondary:
35Q35,  
76D05.  
 }
 \keywords{Semilinear parabolic equations, quasilinear parabolic equations, 
 critical spaces, Navier-Stokes equations, vorticity equations, scaling invariance}

\begin{abstract}
We present a comprehensive theory of critical spaces for the broad class of quasilinear parabolic evolution equations.
The approach is based on  maximal $L_p$-regularity in time-weighted function spaces. 
It is shown that our notion of critical spaces coincides with the concept of scaling invariant spaces in case that
the underlying partial differential equation enjoys a scaling invariance.
Applications to the vorticity equations for the Navier-Stokes problem,
convection-diffusion equations,
the Nernst-Planck-Poisson equations in electro-chemistry,
chemotaxis equations,
the MHD equations, and some other well-known parabolic equations are given.
\end{abstract}

\maketitle

\section{Introduction}
In the last decades there has been an increasing interest in finding critical spaces for nonlinear parabolic partial differential equations. 
There is an extensive literature on this program, but so far a general unified approach seems to be lacking and each equation seems to require its own theory.

As a matter of fact, there is no generally accepted definition 
in the mathematical literature concerning the notion of critical spaces.
One possible definition may be based on the idea of a `largest space of initial data such that the given PDE is well-posed.'
However, this is a rather vague concept that requires additional clarifying information. 
Critical spaces are often introduced as  `scaling invariant spaces,' provided the underlying PDE enjoys a scaling invariance.
A prototype example is given by the Navier-Stokes problem on $\R^d$,
\begin{equation*}
\begin{aligned}
\partial_t u + u\cdot \nabla u -\Delta u +\nabla\pi =0,\quad 
{\rm div}\, u =0, \quad 
          u(0)= u_0,
\end{aligned}
\end{equation*}  
which is invariant under the scaling
$$(u_\lambda (t,x),\pi_\lambda(t,x)):=(\lambda u(\lambda^2 t,\lambda x),\lambda^2 \pi(\lambda^2 t,\lambda x)).$$
In this case one shows that the spaces $L_d(\R^d)$ and $\dot B^{d/q-1}_{qp}(\R^d)$ are scaling invariant for $u$,
and thus are `critical spaces.'

Clearly, this latter concept of `critical space' breaks down as soon as a given equation fails to have a scaling invariance.

In this paper we present a comprehensive theory of critical spaces for the broad class of quasilinear parabolic evolution equations.
Our approach is based on the concept of maximal $L_p$-regularity in time-weighted function spaces. 
In this framework, we introduce the notion of a `critical weight' $\mu_c$ and a corresponding  `critical space' $X_c=X_{\gamma,\mu_c}$.
We will show that 
\begin{enumerate}
\vspace{1mm}
\item $X_{c}$ is, in a generic sense, the largest space of initial data for which the given equation is well-posed.
\vspace{1mm}
\item $X_{c}$ is scaling invariant, provided the given equation admits a scaling.
\end{enumerate}
\vspace{1mm}
The spaces $X_{c}$, thus, encompass and combine the properties mentioned above.
We shall also show that this definition of `critical space' awards us with considerable flexibility in choosing 
an appropriate setting for analyzing a given equation.
For instance, it turns out that the critical spaces $X_{c}$ are independent within the
scale of interpolation-extrapolation spaces associated with a given partial differential equation
(in a sense to be made more precise below).

With our approach, we are able to recover many known results in a unified way,
and on the other side, we will be able to add a variety of new results for some well-known partial differential equations.

\medskip
The concept of `critical weight' was first introduced by Pr\"uss and Wilke in \cite{PrWi17}
and was then applied to the Navier-Stokes equations by Pr\"uss and Wilke \cite{PrWi17, PrWi17a},
and to the quasi-geostrophic equations by Pr\"uss \cite{Pru17}.

\medskip
In this paper, we elaborate on the properties of the critical spaces $X_{c}$ alluded to above.
In addition, we include applications to
the Cahn-Hilliard equations, 
the vorticity equations for the Navier-Stokes problem,
convection-diffusion equations,
the Nernst-Planck-Poisson equations in electro-chemistry,
chemotaxis equations,
and the MHD equations.
 
\bigskip
For the reader's convenience we now state and explain the basic underlying result on 
quasilinear parabolic evolution equations obtained recently in Pr\"uss and Wilke~\cite{PrWi17}.

\medskip
Let $X_0,X_1$ be Banach spaces such that $X_1$ embeds densely in $X_0$, let $p\in(1,\infty)$ and  $1/p<\mu\leq1$.
We consider the following quasilinear parabolic evolution equation
\begin{equation}
\label{qpp}
\dot{u} +A(u)u = F_1(u)+F_2(u),\; t>0,\quad u(0)=u_1.
\end{equation}
The space of initial data will be the real interpolation space $X_{\gamma,\mu} =(X_0,X_1)_{\mu-1/p,p}$, and the state space of the problem is $X_\gamma=X_{\gamma,1}$. Let $ V_\mu\subset X_{\gamma,\mu}$ be open and $u_1\in V_\mu$. 
Furthermore, let $X_\beta=(X_0,X_1)_\beta$, $\beta\in (0,1)$,  
denote  the complex interpolation spaces. We will impose the following assumptions.

\medskip

\noindent
{\bf (H1)} $(A,F_1)\in C^{1-}(V_\mu; \cB(X_1,X_0)\times X_0)$.

\medskip

\noindent
{\bf (H2)} $F_2: V_\mu\cap X_\beta \to X_0$ satisfies the estimate
$$ 
|F_2(u_1)-F_2(u_2)|_{X_0} \leq C \sum_{j=1}^m (1+|u_1|_{X_\beta}^{\rho_j}+|u_2|_{X_\beta}^{\rho_j})|u_1-u_2|_{X_{\beta_j}},
$$
$ u_1, u_2\in V_\mu\cap X_\beta$,
for some numbers $m\in\NN$, $\rho_j\geq 0$, $\beta\in (\mu-1/p,1)$, $\beta_j\in [\mu-1/p, \beta]$, where $C$ denotes a constant which may depend on $|u_i|_{X_{\gamma,\mu}}$.
The case $\beta_j=\mu-1/p$ is only admissible if 
{\bf (H2)} holds with $X_{\beta_j}$ replaced by $X_{\gamma,\mu}$.

\medskip

\noindent
{\bf (H3)} For all $j=1,\ldots,m$ we have
$$ \rho_j( \beta-(\mu-1/p)) + (\beta_j -(\mu-1/p)) \leq 1 -(\mu-1/p).$$

\noindent
Allowing for equality in {\bf (H3)} is not for free and we additionally need to impose the following structural {\bf Condition (S)} on the Banach spaces $X_0$ and $X_1$.

\medskip

\noindent
{\bf (S)} The space $X_0$ is of class UMD. The embedding
$$ {H}^{1}_p(\RR;X_0)\cap L_{p}(\RR;X_1)\hookrightarrow {H}^{1-\beta}_{p}(\RR;X_\beta),$$
is valid for each $\beta\in [0,1]$.

\begin{remark}
By the {mixed derivative theorem}, Condition {\bf (S)} is valid if $X_0$ is of class UMD and if there is an operator $A_{\#}\in \cH^\infty(X_0)$, with domain
${\sf D}(A_\#)=X_1$, and $\cH^\infty$-angle $\phi_{A_\#}^\infty<\pi/2$. We refer to Pr\"uss and Simonett \cite[Chapter 4]{PrSi16}.
\end{remark}

\noindent
The usual {\em solution spaces} in the framework of  {\em maximal $L_p$}-regularity are
$$ u\in H^1_p((0,a);X_0)\cap L_p((0,a),X_1)\hookrightarrow C([0,a];X_{\gamma,1}),$$
where the {\em state space} $X_{\gamma,1}$ is defined by
$$ X_{\gamma,\mu}=(X_0,X_1)_{\mu-1/p,p},$$
with $\mu=1$. Here we want to advertise for {\em time-weighted spaces}, defined by
$$ u\in L_{p,\mu}((0,a);Y) \quad \Leftrightarrow \quad t^{1-\mu} u \in L_p((0,a);Y)),\quad 1\geq\mu>1/p.$$
The corresponding solution classes in the time weighted case are
$$ u\in H^1_{p,\mu}((0,a);X_0)\cap L_{p,\mu}((0,a),X_1)\hookrightarrow C([0,a];X_{\gamma,\mu}).$$
There are several compelling reasons for time weights, among them the following:
\begin{itemize}
\vspace{1mm}
\item Reduced initial regularity,
\vspace{1mm}
\item Instantaneous gain of regularity,
\vspace{1mm}
\item Compactness properties of orbits.
\end{itemize}
\vspace{1mm}
Important is the fact that maximal regularity is independent of $\mu\in (1/p,1]$. 
In the $L_p$-framework, this was first observed by Pr\"uss and Simonett \cite{PrSi04}.

\medskip

In Pr\"uss and Wilke \cite{PrWi17} the following extension of Theorem 2.1 in LeCrone, Pr\"uss and Wilke \cite{LPW14} was obtained.

\begin{theorem}\label{main} Suppose that the structural assumption {\bf (S)} holds, and assume that hypotheses {\bf (H1), (H2), (H3)} are valid. Fix any $u_0\in V_\mu$ such that $A_0:= A(u_0)$ has maximal $L_p$-regularity. Then there  is $a=a(u_0)>0$ and $\varepsilon =\varepsilon(u_0)>0$ with $\bar{B}_{X_{\gamma,\mu}}(u_0,\varepsilon) \subset V_\mu$ such that problem \eqref{qpp} admits a unique solution
$$ u(\cdot, u_1)\in H^1_{p,\mu}((0,a);X_0)\cap L_{p,\mu}((0,a); X_1) \cap C([0,a]; V_\mu),$$
for each initial value $u_1\in \bar{B}_{X_{\gamma,\mu}}(u_0,\varepsilon).$ Furthermore, there is a constant $c= c(u_0)>0$ such that
$$ |u(\cdot,u_1)-u(\cdot,u_2)|_{\EE_{1,\mu}(0,a)} \leq c|u_1-u_2|_{X_{\gamma,\mu}}$$
holds for all $u_1,u_2\in \bar{B}_{X_{\gamma,\mu}}(u_0,\varepsilon)$.
\end{theorem}

\noindent
We call $j$ subcritical if  strict inequality holds in {\bf (H3)}, and critical otherwise. 
As $\beta_j\leq \beta<1$, any $j$ with $\rho_j=0$ is subcritical.
Furthermore,  {\bf (H3)} is equivalent to $ \rho_j\beta+\beta_j -1\leq \rho_j(\mu-1/p)$, hence the minimal value of $\mu$ is given by
$$ \mu_{c} = \frac{1}{p} + \beta -\min_j(1-\beta_j)/\rho_j.$$
We call this value the {\em critical weight}. Thus Theorem \ref{main} shows that we have local well-posedness of \eqref{qpp} for initial values in
the space $X_{\gamma,\mu_{}}$. Therefore, it is meaningful to name this space the {\em critical space} for \eqref{qpp}.

Note that the critical space $X_{\gamma,\mu_c}$ is given by the real interpolation space
$$ X_{\gamma,\mu_c}= (X_0,X_1)_{\mu_c-1/p,p},$$
and $\mu_{c} -\frac{1}{p} =\beta -\min_j(1-\beta_j)/\rho_j$ is independent of $p$. Therefore, the exponent $p$ only shows 
up as a microscopic parameter.

Recall the embeddings
\begin{equation*}
\begin{aligned}
&(X_0,X_1)_{\alpha,p_1}\hookrightarrow (X_0,X_1)_{\alpha,p_2}\hookrightarrow (X_0,X_1)_{\beta,1}, 
\quad \mbox{for } p_1\leq p_2,\; 0<\alpha<\beta<1,\\
&(X_0,X_1)_{\beta,1}\hookrightarrow (X_0,X_1)_\beta\hookrightarrow (X_0,X_1)_{\beta,\infty}.
\end{aligned}
\end{equation*}
The philosophy is to choose $p$ large, say $1/p< 1-\beta$. We then have
$$ X_{\gamma,1} = (X_0,X_1)_{1-1/p,p}\hookrightarrow X_\beta\hookrightarrow X_{\gamma,\mu_c},$$
as $\mu_c-1/p<\beta <1-1/p$. As a consequence, the qualitative theory of K\"ohne, Pr\"uss and Wilke \cite{KPW10} and Pr\"uss, Simonett and Zacher \cite{PSZ09} is available; see also Pr\"uss and Simonett \cite{PrSi16}, Chapter 5.

In Section 2 we give an example which shows that Theorem \ref{main} is optimal in a generic sense 
(but might not be if additional structural properties hold).
This means in specific applications that our theory yields generic lower bounds for the critical space. 
Nevertheless, we find in all PDE applications considered so far that the critical spaces obtained by our theory coincide with the known ones, which in most cases come from (local) scaling invariance. 
Also, a PDE equation can often be considered in a scale of function spaces, 
and then it turns out that the critical spaces are widely independent of this scale.

\medskip

\section{Semilinear parabolic evolution equations with bilinear nonlinearities}
This is a  special case of the {quasilinear theory} presented above, and it encompasses 
many important differential equations in fluid dynamics, physics, and chemistry; for instance, the
Navier-Stokes equations,
vorticity equations for the Navier-Stokes problem,
quasi-geostrophic (subcritical) equations,
convection-diffusion equations,
Nernst-Planck-Poisson equations,
magneto-hydrodynamics, and many more. 

The section is organized as follows.
In Subsection 2.1 we first formulate a result on local well-posedness for equation~\eqref{3.1}
and give a sketch of the proof for the critical case.
The ingredients of the proof for the critical case rely in an essential way on the mixed derivative theorem 
and sharp embedding results for time-weighted Sobolev spaces.
Corollaries~\ref{cor2} and~\ref{cor1} describe conditions for global existence,
while Theorem~\ref{thm5}  contains a result of Serrin type which states that global existence is equivalent to an integral a priori bound.
In Subsection 2.2 we show by means of a counterexample that
the critical spaces $X_{\gamma,\mu_c}$ identified in Theorem~\ref{thm1} are - in a generic sense - the largest spaces of initial data for
which equation~\eqref{3.1} is well-posed.
In Subsection 2.3 it is shown that the (homogeneous versions of) critical spaces are scaling invariant, provided 
equation~\eqref{3.1} admits a scaling.
Finally, in Subsection~2.4 we show that the critical spaces are invariant with respect to the extrapolation-interpolation scale associated with
the operator $A$ in equation~\eqref{3.1}.
Here we would like to mention that  similar results are also true (with appropriate modifications) for
semilinar parabolic equations with multilinear nonlinearities. 

\subsection{Local and global existence of solutions}
Let $X_0$ be a UMD-space, $X_1\hookrightarrow X_0$ densely, $A:X_1\to X_0$ bounded and such that $A\in\mathcal{BIP}(X_0)$ 
with power angle $\theta_A<\pi/2$, and let $p\in(1,\infty)$.  Consider the semilinear parabolic evolution equation
\begin{equation}
\label{3.1}
\partial_t u +Au = G(u,u),\; t>0,\quad u(0)=u_0.
\end{equation}
Here $G:X_\beta\times X_\beta \to X_0$ is bilinear and bounded, with
$$ X_\beta = (X_0,X_1)_\beta = {\sf D}(A^\beta),$$
for some $\beta\in [0,1)$.

Our main result for \eqref{3.1} reads as follows.
\goodbreak
\begin{theorem} 
\label{thm1}
 Assume $p\in (1,\infty)$, $\mu\in (1/p,1]$, $\beta\in (\mu-1/p,1)$ and
\begin{equation}
\label{3.2} 
2\beta-1 \leq \mu-1/p.
\end{equation}
Then for each initial value $u_0\in X_{\gamma,\mu}$ there is $a=a(u_0)>0$ and a unique solution of (\ref{3.1}) in the class
$$ u\in H^1_{p,\mu}((0,a);X_0)\cap L_{p,\mu}((0,a),X_1)\hookrightarrow C([0,a];X_{\gamma,\mu}).$$
The solution exists on a maximal time interval $[0,t_+(u_0))$ and depends continuously on the data and moreover satisfies
$$ u\in H^1_{p,loc}((0,t_+);X_0)\cap L_{p,loc}((0,t_+),X_1)\hookrightarrow C((0,t_+);X_{\gamma,1}).$$
Hence it regularizes instantly, provided $\mu<1$.
\end{theorem}

\noindent
Hence, $\mu$ is {\em subcritical} for \eqref{3.1} if strict inequality holds in \eqref{3.2}, and {\em critical} otherwise. The case 
$\beta\leq 1/2$ is always subcritical, and if $\beta>1/2$ then $\mu_c:= 2\beta -1+1/p$ is the {critical weight} and
$$ X_{\gamma,\mu_c} = D_A(2\beta-1,p)$$
is the {critical space} for (\ref{3.1}). 
We observe again that $p$ appears only as a microscopic parameter.

\bigskip
\noindent
{\bf Main idea for the proof of Theorem \ref{thm1}}.
The semilinear case \eqref{3.1} with bilinear nonlinearity is considerably simpler than the quasilinear case \eqref{qpp},
and we give here an outline of the proof for this simpler case.
The arguments then are rather short and employ the contraction mapping principle and several sharp embeddings.
We focus on the critical case, i.e. we assume that $\beta>1/2$ and we choose $\mu=\mu_c=2\beta -1 +1/p$.

The fixed point equation reads
$$ v = Tv:=e^{-At}* G(u,u), \quad v=u-u_*, \quad u_*= e^{-At}u_0,$$
in a ball $\BB_r:= \bar{B}_{{_0\EE}_\mu(a)}(0,r)$ in the space
$${_0\EE}_\mu(a) = {_0H}^1_{p,\mu}((0,a);X_0)\cap L_{p,\mu}((0,a);X_1),$$
where $r$ and $a\in (0,a_0)$, for some fixed $a_0$, are at our disposal. 
Here  ${_0H}^1_{p,\mu}((0,a);X_0)$ denotes the functions $v\in H^1_{p,\mu}((0,a);X_0)$ with $v(0)=0$.

We estimate as follows, with  maximal regularity constant $M=M(a_0)\geq1$ and $1-\mu= 2(1-\tau)$,
\begin{equation*}
\begin{aligned}
|Tv|_{{_0\EE_\mu(a)}} &\leq M |G(u,u)|_{L_{p,\mu}(X_0)} \leq MC | |u|^2_{X_\beta}|_{L_{p,\mu}}\\
& =MC |u|^2_{L_{2p,\tau}(X_\beta)}\leq 2MC [ |u_*|^2_{L_{2p,\tau}(X_\beta)}+|v|^2_{L_{2p,\tau}(X_\beta)}].
\end{aligned}
\end{equation*}
As $u_0\in D_A(2\beta-1,p)\subset D_A(2\beta-1,2p)=D_A(\beta -1 + \tau -1/2p, 2p)$ we obtain
$$u_*\in L_{2p,\tau}((0,a);X_\beta),$$
see \cite[Proposition 3.4.3]{PrSi16}. 
The first term $2MC|u_*|^2$ can be made small, say smaller than $r/2$, by choosing $a\in (0,a_0)$ small.
Next we have the embeddings
\begin{equation*}
\begin{aligned}
{_0\EE}_{\mu}(a) &\hookrightarrow {_0H}^{1-\beta}_{p,\mu}((0,a); X_\beta) \quad \mbox{(mixed derivative theorem)}\\
& \hookrightarrow L_{2p,\tau}((0,a);X_\beta)\quad \mbox{(Sobolev embedding)},
\end{aligned}
\end{equation*}
as the Sobolev indices for these spaces are the same, 
$$ 1-\beta -1/p -(1-\mu) = -1/2p-(1-\tau)  \quad \mbox{\Black{as}} \quad 2\beta-1= \mu -1/p.$$
{We refer to \cite[Corollary 1.4]{MeVe12} for embedding results in weighted Bessel-potential spaces.}
Note that the  embedding constants do not depend on $a>0$.
So choosing $r>0$ small enough, the remaining term $2MC|v|^2$
will also be small, say smaller than $r/2$. This shows that $T:\BB_r\to \BB_r$ is a self-map. The contraction property is proved in a similar way, by the estimate
\begin{equation*}
\begin{aligned}
|Tv_1-Tv_2|_{{_0\EE}(a)}&\leq M|G(u_1,u_1)-G(u_2,u_2)|_{L_{p,\mu}(X_0)}\\
&\le  MC\Big(|u_1|_{L_{2p,\tau}(X_\beta)}+|u_2|_{L_{2p,\tau}(X_\beta)}\Big) |v_1-v_2|_{L_{2p,\tau}(X_\beta)}.
\quad\square
\end{aligned}
\end{equation*}

\medskip
\noindent
Instead of choosing $a>0$ small, we may instead, for a given $a>0$, choose $r>0$ small enough to obtain a unique solution on $(0,a)$.
If $0\in\rho(A)$ then the maximal regularity constant is independent of $a$ and $r>0$ may be chosen uniformly in $a>0$, 
to obtain global existence and exponential stability. 
\begin{corollary} 
\label{cor2}  
Let the assumptions of Theorem \ref{thm1} hold. Then
\begin{enumerate}
\item[{\bf (i)}] For any given $a>0$ there is $r=r(a)>0$ such that the solution of (\ref{3.1}) exists on $[0,a]$, 
whenever $|u_0|_{X_{\gamma,\mu}}\leq r$.
\vspace{1mm}
\item[{\bf (ii)}] If $0\in\rho(A)$ then $r>0$ is independent of $a$.
\vspace{1mm}
\item[{\bf (iii)}] If $0\in\rho(A)$  and $1/p<1-\beta$, then the trivial solution of (\ref{3.1}) is exponentially stable in the state space $X_{\gamma,1}$. Moreover, there is $r_0>0$ such that the solution $u(t)$ of (\ref{3.1}) converges exponentially to zero in $X_{\gamma,1}$, provided $|u_0|_{X_{\gamma,\mu}}\leq r_0$.
\end{enumerate}
\end{corollary}

\medskip
\goodbreak
\noindent
Concerning {\em conditional global existence} we can prove the following result.
\begin{corollary} 
\label{cor1}  
The local solution of Theorem \ref{thm1} exists globally, provided
\begin{enumerate}
\item[{\bf (i)}] $u([0,t_+))\subset X_{\gamma,\mu}$ is bounded in the subcritical case; or
\vspace{1mm}
\item[{\bf (ii)}] $u([0,t_+))\subset X_{\gamma,\mu_c}$ is relatively compact in the critical case.
\end{enumerate}
\end{corollary}
\begin{proof}
{\bf (i)} Suppose $\mu$ is subcritical, i.e. $\mu-1/p>2\beta-1$, and let
$$\alpha:=\frac{\beta -(\mu-1/p)}{1-(\mu-1/p)}. $$
By interpolation theory we obtain
$(X_{\gamma,\mu}, X_1)_{\alpha,1}=(X_0,X_1)_{\beta,1}\hookrightarrow X_\beta.$
Hence,
$$ |G(u,u)|_{X_0}\le C |u|^2_{X_\beta}\le C |u|^{2\alpha}_{X_1} |u|^{2(1-\alpha)}_{X_{\gamma,\mu}}.$$
To establish global existence we may assume w.l.o.g that $u_0\in X_{\gamma,1}$.
Let $u$ be the unique solution of \eqref{3.1} with initial value $u_0\in X_{\gamma,1}$, defined on
its maximal existence interval $[0,t_+(u_0))$ and suppose that $t_+:=t_+(u_0)<\infty$.
Then we obtain for any $a\in (0,t_+)$ by means of maximal regularity
\begin{equation*}
\begin{aligned}
|u|_{\EE_1(a)}
&\le C_1 |u_0|_{X_{\gamma,1}} + M |G(u,u)|_{L_p((0,a);X_0)}\\
& \le C_1 |u_0|_{X_{\gamma,1}} + MC |u|^{2(1-\alpha)}_{L_\infty((0,t_+); X_{\gamma,\mu})} |u|^{2\alpha}_{\EE_1(a)} \\
&\le C_1 |u_0|_{X_{\gamma,1}} + C_2 |u|^{2\alpha}_{\EE_1(a)},              
\end{aligned}
\end{equation*}
where $M=M(t_+)$ is the constant of maximal regularity.
As $2\alpha<1$, we conclude that $|u|_{\EE_1(a)}$ is bounded,  uniformly in $a\in (0,t_+)$.
Therefore, $|u|_{\EE_1(t_+)}$ is bounded, implying that $u\in C([0,t_+]; X_{\gamma,1})$.
Hence, the solution can be continued beyond $t_+$, a contradiction.
\medskip\\
{\bf (ii)} 
Let $u_0\in X_{\gamma,\mu_c}$ be given and suppose that $t_+=t_+(u_0)<\infty$.
Then the set $\Gamma:= \overline{u([0,t_+))}$ is compact in $X_{\gamma,\mu_c}$.
It follows from Theorem~\ref{thm1} and a covering argument that there is  
$\delta>0$ such that equation \eqref{3.1} has for each 
$v_0\in \Gamma$ a unique solution $v\in \EE_{1,\mu_c}(\delta)$.
But this implies that the solution $u$ can be continued beyond $t_+$, leading once more to a contradiction.
\end{proof}

Next  we prove  a result of Serrin type which states that global existence is equivalent to an integral a priori bound.

\begin{theorem} 
\label{thm5} 
Let $p\in (1,\infty)$, $\beta>1/2$ and $\mu :=2\beta -1 +1/p\leq 1$ the critical weight.
 Assume $u_0\in X_{\gamma,\mu}$, and let $u$ denote the unique solution of \eqref{3.1} 
 defined by Theorem \ref{thm1}, with maximal interval of existence $[0,t_+)$. Then
\begin{enumerate}
\item[{\bf (i)}] $u\in L_p((0,a);X_\mu)$, for each $a<t_+$.
\vspace{1mm}
\item[{\bf (ii)}] If $t_+<\infty$ then $ u\not\in L_p((0,t_+);X_\mu)$.
\end{enumerate}
In particular, the solution exists globally if $u\in L_p((0,a);X_\mu)$ for any finite number~$a$ with $a\le t_+$.
\end{theorem}
\noindent
We remind that $X_\mu=(X_0,X_1)_\mu$ denote the complex interpolation spaces.
It is interesting to observe that the 
Sobolev indices of the spaces $\EE_\mu(a)$, $L_p((0,a);X_\mu)$ and $C([0,a];X_{\gamma,\mu})$ are all the same, given by $\mu-1/p$.

\begin{proof}
{\bf (i)} Let $a\in (0,t_+)$ be fixed. 
By the mixed derivative theorem of Sobolevskii and {Sobolev embedding in weighted spaces, 
see for instance \cite[Corollary 1.4]{MeVe12},} we have
$$ \EE_\mu(a)= H^1_{p,\mu}((0,a); X_0)\cap L_{p,\mu}((0,a);X_1)\hookrightarrow H^{1-\mu}_{p,\mu}((0,a);X_\mu)\hookrightarrow L_p((0,a);X_\mu).$$

{\bf (ii)} 
From the mixed derivative theorem, Proposition \ref{interpol} in the appendix, and {Sobolev embedding in weighted spaces} follows
\begin{align*}
(L_{p}((0,a);X_\mu), \EE_\mu(a))_{1/2} &\hookrightarrow (L_{p}((0,a);X_\mu), H^{1-\alpha}_{p,\mu}((0,a);X_\alpha))_{1/2}\\
&= H^{(1-\alpha)/2}_{p,(1+\mu)/2}((0,a);X_{(\mu+\alpha)/2}) \hookrightarrow L_{2p,\tau}((0,a); X_\beta),
\end{align*}
where $\alpha=1-1/p$ and  $2(1-\tau)=1-\mu$, i.e.\ $ \tau =(1+\mu)/2$. \\

Suppose $t_+<\infty$ and let $a_0\in (0,t_+)$ be fixed.
Employing the interpolation inequality, the quadratic estimate $|G(u,u)|_{X_0}\leq C_1 |u|_{X_\beta}^2$ implies
\begin{equation*}
 |G(u,u)|_{L_{p,\mu}((a_0,a);X_0)} \leq C_1 |u|^2_{L_{2p,\tau}((a_0,a);X_\beta)}
 \leq C_2 |u|_{L_p((a_0,a);X_\mu)} |u|_{\EE_\mu(a_0,a)},
\end{equation*}
where the constant $C_2$ is independent of $a\in (a_0,t_+)$.
Let $M$ be the constant of maximal regularity for the interval $[0,t_+)$,
and let $\eta=1/(2MC_2)$.
We choose $t_0\in (a_0,t_+)$ sufficiently close to $t_+$ such that $|u|_{L_p((t_0,t_+);X_\mu)}\le \eta$.
Then by maximal regularity  we obtain
\begin{equation*}
 |u|_{\EE_{\mu}(t_0,a)} \leq M\big(|u(t_0)|_{X_{\gamma,\mu}} + C_2\eta |u|_{\EE_{\mu}(t_0,a)}\big).
\end{equation*}
By the definition of $\eta$, this yields 
$$|u|_{\EE_{\mu}(t_0,a)} \leq 2M|u(t_0)|_{X_{\gamma,\mu}}$$
for any $a\in (t_0,t_+)$.
Therefore, $u|_{(t_0,t_+)}\in \EE_1(t_0,t_+)\hookrightarrow C([t_0,t_+],X_\mu)$.
This implies that the solution $u$ can be continued beyond $t_+$, leading to a contradiction. 
\end{proof}

\subsection{A Counterexample}
We want to show that Theorem \ref{thm1}, and hence also Theorem \ref{main}, 
are optimal in the sense that, generically, \eqref{3.1} is not well-posed in spaces 
strictly larger than the critical space $X_{\gamma,\mu_c}$.

\medskip

\noindent
{\bf (i)}\, We choose a sequence $a_k>0$ with $a_k\to\infty$, and set 
\begin{equation*}
 X_0 =l_2(\NN),\quad (Au)_k= a_k u_k,\quad X_1={\sf D}(A)= l_2(\NN;a_k).
 \end{equation*}
This operator is selfadjoint and positive definite in the Hilbert space $X_0$. Furthermore, 
{by complex interpolation in weighted $l_2$-spaces} we have
$X_\beta = l_2(\NN;a_k^\beta)={\sf D}(A^\beta)$, 
for all $\beta\geq0$. Next we define a symmetric bilinear operator $G$ by means of
$$G:X_\beta\times X_\beta \to X_0,\quad G(u,v)_k := a_k^{2\beta} u_k v_k.$$
Obviously, $G$ is bilinear, and it also bounded by the Cauchy-Schwarz inequality, as $l_2(\NN)\hookrightarrow l_4(\NN)$.

Consider the evolution equation
\begin{equation} 
\label{1}
\partial_t u +Au =G(u,u),\; t>0,\quad u(0)=u_0,
\end{equation}
in $X_0$. Then we are in the situation of Theorem \ref{thm1} hence we have local well-posedness in $L_{p,\mu}$ for initial data in $X_{\gamma,\mu}=D_A(\mu-1/p,p)$,
for all $\mu\geq\mu_c$, where the critical weight $\mu_c$ is defined by
$$ 2\beta-1 = \mu_c-1/p.$$
Below we require $\beta>1/2$ and $\beta \leq 1-1/2p$. By an appropriate choice of the coefficients $a_k$, we want to show that this problem is ill-posed for any weight $1/p<\mu<\mu_c$, showing that Theorem \ref{thm1} cannot be improved, 
and hence the condition \eqref{3.2} is sharp.

\medskip

\noindent
{\bf (ii)} In components, problem \eqref{1}  reads
\begin{align*}
\partial_t u_k +a_k u_k &= a_k^{2\beta} u_k^2,\quad t>0,\; k\in\NN,\\
u_k(0) &= u^0_k.
\end{align*}
This system can be solved explicitly. In fact, by means of the scaling $v_k(t)= a_k^{2\beta-1} u_k(t/a_k)$, the system transforms to
\begin{align*}
\partial_t v_k +v_k &=  v_k^2,\quad t>0,\; k\in\NN,\\
v_k(0) &= v^0_k= a_k^{2\beta-1}u_k^0.
\end{align*}
Solving the bi-stable equation
$$ \partial_t w +w = w^2,\; t>0,\quad w(0)=w_0,$$
we get
$$w(t)= \frac{w_0 e^{-t}}{ 1-w_0 +e^{-t}w_0} = \frac{w_0 e^{-t}}{ 1-w_0(1- e^{-t})}.$$
So $w(t)$ exists globally to the right if $w_0\leq1$ and a blow up occurs if $w_0>1$.
This yields for $v_k(t)$ the formula
$$v_k(t) = \frac{v_k^0 e^{-t}}{1-v_k^0(1-e^{-t})},$$
hence inverting the scaling
$$ u_k(t) = \frac{ u_k^0 e^{-a_k t}}{1-a_k^{2\beta-1} u_k^0(1-e^{-a_kt})},\quad t>0,\; k\in\NN.$$
This implies that whenever a solution of class $L_{p,\mu}$, i.e.
$$u\in H^1_{p,\mu}((0,a); X_0)\cap L_{p,\mu}((0,a);X_1)$$
exists, then the initial value $u_0$ must satisfy
\begin{equation}\label{nec-cond}
\overline{\lim}_{k\to\infty} a_k^{2\beta-1}u_k^0 \leq 1.
\end{equation}
{\bf (iii)}
Now suppose that $u_k^0\geq 0$ for all $k$, and set $a_k = 2^k$. If \eqref{nec-cond} holds, then $|u_k^0|\leq c 2^{k(1-2\beta)}$, for some constant $c>1$, hence
$$ |A^s u_0|^2_{X_0} = \sum_{k\geq1} 2^{2ks}|u_k^0|^2 \leq c\sum_{k\geq1} 2^{2k(s-2\beta +1)} <\infty,$$
for $s<2\beta-1$. This implies that whenever a solution of class $L_{p,\mu}$ exists, then $u_0\in {\sf D}(A^s)$ for all $s<2\beta-1$.

We thus find ourselves in the following situation:
If there is some $\mu\in (1/p,\mu_c)$ and some initial value 
$u_0\in X_{\gamma,\mu}=(X_0,X_1)_{\mu-1/p,p}$ such that 
\eqref{1} has a solution in the class $H^1_{p,\mu}((0,a);X_1)\cap L_{p,\mu}((0,a);X_0)$ then
$u_0$ must be in $\bigcap_{s<2\beta-1}{\sf D}(A^s)$,
which is a space strictly contained in $(X_0,X_1)_{\mu-1/p,p}$.
Hence the assumption that the problem is well-posed in the class $L_{p,\mu}$ for $\mu<\mu_c$ leads to a contradiction.
\medskip

\noindent
{\bf (iv)} Note that in this example we have global existence in $L_{p,\mu}$ for any $1\geq\mu >1/p$ if the initial data are non-positive.
In that case we have the estimates
$$ |u_k(t)|\leq |u_k^0|e^{-a_kt},\; t\geq0,\; k\in\NN,$$
hence the solution of the nonlinear problem is dominated by the semigroup $e^{-At}$. This shows that in this case we have well-posedness for all initial data, even for $u_0\in X_0$.

We may modify the example in several ways. Replacing the nonlinearity by $-w^3$ we obtain global $L_{p,\mu}$-solutions for all initial values $u_0\in D_A(\mu-1/p,p)$, as well as global exponential stability. And if we replace the nonlinearity by $w|w|$ the signs of the initial conditions are not important.

Such examples show that it may very well happen that \eqref{3.1} is still well-posed in spaces of initial data larger than $X_{\gamma,\mu_c}$, but this involves more structural properties of the equation under consideration.

\subsection{Scaling Invariance}
 Let $A\in\mathcal{BIP}(X_0)$ and $G:\dot{X}_\beta\times \dot{X}_\beta\to X_0$ bounded bilinear.
Here $\dot{X}_\beta$ means the completion of ${\sf D}(A^\beta)$ in the homogeneous norm $|A^\beta\cdot|$. In a similar way we define the spaces 
$\dot{D}_A(\alpha,p)$ as the completions of $D_A(\alpha,p)$ in the homogeneous norms 
$$ |x|_{\dot{D}_A(\alpha,p)} = [\int_0^\infty |r^\alpha A(r+A)^{-1}x|^pdr/r]^{1/p}.$$
We begin with a definition.
\begin{defi}
A family of operators $\{T_\lambda\}_{\lambda>0}\subset \cB(X_0)\cap \cB(X_1)$ is a called a {\bf scaling } for \eqref{3.1} if the following conditions hold.\\
{\bf (i)} $AT_\lambda =\lambda T_\lambda A$, for all $\lambda>0$;\\
{\bf (ii)} $\lambda T_\lambda G(x,x) = G(T_\lambda x, T_\lambda x)$, for all $\lambda>0, \; x\in \dot{X}_\beta$;\\
{\bf (iii)} there are constants $c>0$ and $\delta\in \R$ such that
$$ c^{-1}\lambda^{-\delta} |x|\leq |T_\lambda x|\leq c\lambda^{-\delta}|x|,\quad \lambda>0,\; x\in X_0.$$
\end{defi}
\noindent
It is easy to see that if $u(t)$ is a solution of \eqref{3.1}, then $u_\lambda(t) = T_\lambda u(\lambda t)$ is again a solution of \eqref{3.1} with initial value
$u_\lambda(0) = T_\lambda u_0$.

\medskip

\noindent
From {\bf (i)} we obtain
$$ (z-A) T_\lambda = T_\lambda z - \lambda T_\lambda A = \lambda T_\lambda (z/\lambda -A),$$
hence
$$ (z-A)^{-1} T_\lambda = \frac{1}{\lambda} T_\lambda(z/\lambda -A)^{-1},$$
and so again by {\bf (i)}
\begin{equation}
\label{AT}
 A(z-A)^{-1}T_\lambda = T_\lambda A(z/\lambda -A)^{-1}.
\end{equation}
This implies
\begin{equation*}
\begin{aligned} |T_\lambda x|_{\dot{D}_A(\alpha,p)} &= [ \int_0^\infty |r^\alpha A(r+A)^{-1}T_\lambda x|^pdr/r]^{1/p} \\
&= \lambda^\alpha [ \int_0^\infty |r^\alpha T_\lambda A(r+A)^{-1} x|^pdr/r]^{1/p},
\end{aligned}
\end{equation*}
which by (iii) yields
$$ c^{-1}\lambda^{\alpha-\delta} |x|_{\dot{D}_A(\alpha,p)}\leq |T_\lambda x|_{\dot{D}_A(\alpha,p)}\leq  c\lambda^{\alpha-\delta} |x|_{\dot{D}_A(\alpha,p)}.$$
In particular, the norm of $T_\lambda x$ in $\dot{D}_A(\alpha,p)$ is, up to constants, independent of $\lambda>0$ if and only if $\alpha=\delta$. Thus the spaces
$\dot{D}_A(\delta,p)$ are {\em scaling invariant} for the scaling $T_\lambda$.
From \eqref{AT} we also obtain, with an appropriate contour $\Gamma$
$$ A^\beta T_\lambda x = \frac{1}{2\pi i}\int_\Gamma z^\beta A(z-A)^{-1} T_\lambda x dz/z= \lambda^\beta T_\lambda A^\beta x.$$
Next we employ {\bf (ii)} and {\bf (iii)} to obtain
\begin{align*}
c^{-1} \lambda^{1-\delta}|G(x,x)|&\leq |\lambda T_\lambda G(x,x)|= |G(T_\lambda x, T_\lambda x)|\\
&\leq C |A^\beta T_\lambda x|^2 = C |\lambda^\beta T_\lambda A^\beta x|^2 \leq Cc^2 \lambda^{2\beta-2\delta}|A^\beta x|^2.
\end{align*}
This implies the inequality
$$ |G(x,x)|\leq Cc^3 \lambda^{2\beta-1-\delta}|A^\beta x|^2, \quad \lambda>0,\; x\in X_\beta.$$
If $G$ is nontrivial, i.e.\ $G(x_0,x_0)\neq 0$ for some $x_0\in \dot{X}_\beta$, this implies $\delta = 2\beta-1$, otherwise we obtain a contradiction by letting either $\lambda\to\infty$ or $\lambda\to 0$.

\medskip

\noindent
This shows that the critical weight $\mu_c$ satisfies
$ \mu_c -1/p =2\beta-1 = \delta,$
thereby proving that the critical spaces for \eqref{3.1} are scaling invariant.

\medskip

\noindent
{\bf Example.} {\em The Navier-Stokes equations in $\RR^d$.}\\
Consider the Navier-Sokes problem in $\RR^d$, which reads
\begin{equation*}
\begin{aligned}
\dot{u} -\Delta u +\nabla \pi&= -u\cdot\nabla u &&  \mbox{in }\; \RR^d,\\
{\rm div}\, u &=0  && \mbox{in }\; \RR^d,\\
u(0)&=u_0 &&  \mbox{in }\; \RR^d.
\end{aligned}
\end{equation*}
This problem has the scaling invariance given by 
$$ u_\lambda(t,x) = \sqrt\lambda u(\lambda t,\sqrt\lambda x),\quad \pi_\lambda(t,x) = \lambda \pi(\lambda t, \sqrt\lambda x),$$
i.e.\ we have $T_\lambda u(x) =\sqrt\lambda u(\sqrt\lambda x)$. In the $L_q$-setting, it is easy to compute the scaling number $\delta$ in {\bf (iii)} to the result $\delta= (d/q-1)/2$. Therefore, with
$$\dot{X_\beta}= \dot{H}^{2\beta}_{q,\sigma}(\RR^d),\quad \dot{D}_A(\alpha,p) =\dot{B}_{qp,\sigma}^{2\alpha}(\RR^d)$$
as well as $\beta = (d/q + 1)/4$ we obtain that the critical spaces are the scaling invariant spaces $\dot{B}^{d/q-1}_{qp,\sigma}(\RR^d)$.

\subsection{Independence of the Scale}
Consider the complex interpolation-extrapolation scale $(X_s,A_s)$ 
generated by $A_0:=A\in \mathcal{BIP}(X_0)$ with $X_1={\sf D}(A)$, where we assume w.l.o.g that $A$  is invertible.
Suppose that for some $s_0>0$ we have
$$G:X_{\beta-s/2}\times X_{\beta-s/2} \to X_{-s}\quad \mbox{bilinear bounded}, \quad s\in[0,s_0]$$
We claim that the critical spaces for \eqref{3.1} are independent of $s$. Assume first that $\beta>1/2$. Then we find 
$$\mu_c^0 -1/p= 2\beta -1$$
for the critical weight in case $s=0$ and solutions in the class
$$ u\in H^1_{p,\mu_c^0}(J;X_0) \cap L_{p,\mu_c^0}(J;X_1),$$
for initial values in the critical space $ (X_0, X_1)_{2\beta-1,p}$.

\smallskip
Next  fix any $s\in (0,s_0]$ and set $X_0^{\sf w}=X_{-s}$, $A^{\sf w}=A_{-s}$ and $X_1^{\sf w} =X_{1-s}$. 
Then we have with $\beta^{\sf w}=\beta + s/2$
$$ G:X_{\beta^{\sf w}}^{\sf w}\times X_{\beta^{\sf w}}^{\sf w}\to X^{\sf w}_0\quad \mbox{is bilinear and bounded},$$
hence with the critical weight
$$ \mu_c^s -1/p= 2\beta^{\sf w} -1 = 2\beta -1 +s =\mu_c^0-1/p+s$$
we obtain solutions 
$$ u\in H^1_{p,\mu_c^s}(J;X_0^{\sf w}) \cap L_{p,\mu_c^s}(J;X_1^{\sf w}),$$
for initial values $u_0 \in (X_0^{\sf w}, X_1^{\sf w})_{\mu_c^{\sf w}-1/p,p}$.
But by \eqref{re-int}
$$ (X_0^{\sf w}, X_1^{\sf w})_{\mu_c^{\sf w}-1/p,p}= (X_0^{\sf w}, X_1^{\sf w})_{2\beta^{\sf w}-1,p}= (X_{-s},X_{1-s})_{2\beta -1+s,p}=(X_0,X_1)_{2\beta-1,p},$$
which shows the invariance of the critical spaces w.r.t.\ $s\in[0,s_0]$.

\medskip

\noindent
{\bf Remark.} There is the restriction $\beta^{\sf w}<1$, which means $s<2-2\beta$. This yields an upper bound for $s$. On the other hand, if $s=0$ is subcritical, i.e\ if $\beta\leq 1/2$, the problem will be critical in $X_{-s}$ provided $2\beta-1+s>0$ which means $ s> 1-2\beta$.
Thus we have the window
$$1-2\beta <s<2-2\beta$$ for the best choice of $s$. 

\section{Examples of scalar parabolic equations}
\noindent
In this section we consider three scalar parabolic equations. 
We begin with a very classical one, namely with a famous problem studied by Fujita and later on by  Weissler. 
\medskip

\noindent
{\bf Example 1.} Let $\Omega$ be a bounded domain of class $C^{2}$ in $\RR^d$, and consider the Dirichlet problem
\begin{equation}
\label{eq:Ex1}
\begin{aligned}
\partial_t u -\Delta u &= |u|^{\kappa-1} u &&\mbox{in}\;\; \Omega, \\
u&=0  && \mbox{on}\;\; \partial\Omega,\\
u&=u_0  && \mbox{in}\;\; \Omega,\
\end{aligned}
\end{equation}
where $\kappa>1$.

\medskip
We consider this example in strong and weak functional analytic settings, to be made precise below. Note that the (local) scaling invariance is given by
$$u_\lambda(t,x)= \lambda^{\frac{1}{\kappa-1}}u(\lambda t,\lambda^{\frac{1}{2}}x).$$
\emph{Strong setting:} Let $X_0=L_q(\Omega)$, $1<q<\infty$,
$$X_1:=\{u\in H_q^2(\Omega):u|_{\partial\Omega}=0\}$$
and define an operator $A:X_1\to X_0$ by $Au:=-\Delta u$. With this choice, it holds that
\begin{equation*}
X_\beta=(X_0,X_1)_{\beta}
=
\left\{
\begin{aligned}
& \{u\in H_q^{2\beta}(\Omega):u|_{\partial\Omega}=0\} &&\text{for} &&  \beta\in (1/2q,1],\\
& H_q^{2\beta}(\Omega) &&\text{for} &&\beta\in [0,1/2q).
\end{aligned}
\right.
\end{equation*}
Define $F(u)=|u|^{\kappa-1}u$ for $u\in X_\beta$. Then 
$$|F(u)|_{X_0} = |u|_{\kappa q}^\kappa\le C|u|_{X_\beta}^\kappa,$$
and by the fundamental theorem of calculus
$$|F(u)-F(\bar{u})|_{X_0}\le C(|u|_{X_\beta}^{\kappa-1}+|\bar{u}|_{X_\beta}^{\kappa-1})|u-\bar{u}|_{X_\beta},$$
provided $H_{q}^{2\beta}(\Omega)\hookrightarrow L_{\kappa q}(\Omega)$, which is the case for 
$$\beta=\frac{d}{2q}\left(1-\frac{1}{\kappa}\right),$$
hence $q>\frac{d}{2}(1-\frac{1}{\kappa})$, by the constraint $\beta<1$. Setting $\rho_1=\kappa-1$, $\beta_1=\beta$ in \textbf{(H2)}, the critical weight $\mu_c$ is given by
$$\mu_c=\frac{1}{p}+\frac{\kappa\beta-1}{\kappa-1}=\frac{1}{p}+\frac{d}{2q}-\frac{1}{\kappa-1},$$
which results from \textbf{(H3)}. The condition $\mu_c>1/p$ is then equivalent to $\beta>1/\kappa$, hence 
\begin{equation}\label{eq:Ex1q}
\frac{d(\kappa-1)}{2\kappa}<q<\frac{d(\kappa-1)}{2}.
\end{equation}
As $q>1$, this means $d(\kappa-1)/2>1$, hence $\kappa>1+2/d$.
Since $\mu_c\le 1$, we obtain the additional relation
$$\frac{2}{p}+\frac{d}{q}\le \frac{2\kappa}{\kappa-1}$$
for $p,q\in (1,\infty)$ satisfying \eqref{eq:Ex1q}. In the sequel, for $s\in (0,1)\backslash\{1/2q\}$, let
\begin{equation*}
{_0}B_{qp}^{2s}(\Omega)=(X_0,X_1)_{s,p}=
\left\{
\begin{aligned}
& \{u\in B_{qp}^{2s}(\Omega):u|_{\partial\Omega}=0\} &&\text{for} && s\in(1/2q,1),\\
& B_{qp}^{2s}(\Omega) &&\text{for} &&s\in (0,1/2q).
\end{aligned}
\right.
\end{equation*}
Then, the critical  space is given by 
$$X_{\gamma,\mu_c}=(X_0,X_1)_{\mu_c-1/p,p}={_0}B_{qp}^{d/q-2/(\kappa-1)}(\Omega).$$
Applying Theorem \ref{main}, we obtain the following result.
\begin{theorem} 
\label{thm:Ex1strong}
Let $\kappa>1+2/d$, $p\in (1,\infty)$ and let $q\in (1,\infty)$ satisfy \eqref{eq:Ex1q} such that $2/p+d/q\le 2\kappa/(\kappa-1)$.

Then, for each {$u_0\in {_0B}^{d/q-2/(\kappa-1)}_{qp}(\Omega)$}, problem \eqref{eq:Ex1}
admits a unique solution
$$ u \in H^1_{p,\mu_c}((0,a);L_q(\Omega))\cap L_{p,\mu_c}((0,a); H_q^2(\Omega)),$$
for some {$a>0$}, with critical weight {$\mu_c = 1/p+ d/2q-1/(\kappa-1)$}. The solution exists on a maximal interval {$(0,t_+(u_0))$} and depends continuously on {$u_0$}.
In addition,
$$ u \in C([0,t_+); {_0B}^{d/q-2/(\kappa-1)}_{qp}(\Omega))\cap C((0,t_+);{_0 B}^{2(1-1/p)}_{qp}(\Omega)),
$$
i.e.\ the solutions regularize instantly if {$2/p +d/q<2\kappa/(\kappa-1)$}.
\end{theorem}
\emph{Weak setting:}
Here, we employ the theory of interpolation-extrapolation scales, see the Appendix. Let $A_0:=A$ with domain $X_1$ and $(X_\alpha,A_\alpha)$, $\alpha\in\R$, the interpolation-extrapolation scale $(X_\alpha,A_\alpha)$, $\alpha\in\R$, generated by $(X_0,A_0)$ with respect to the complex interpolation functor.  Let further $1/q+1/q'=1$, $X_0^\sharp=L_{q'}(\Omega)$, $A_0^\sharp=-\Delta|_{X_0^\sharp}$ with domain 
$$X_1^\sharp=\{u\in H_{q'}^2(\Omega):u|_{\partial\Omega}=0\}.$$ 
Then $(X_0^\sharp,A_0^\sharp)$ generates the dual scale $(X_\alpha^\sharp,A_\alpha^\sharp)$. In particular, it holds
$$X_{1/2}=(X_0,X_1)_{1/2}=\{u\in H_q^1(\Omega):u|_{\partial\Omega}=0\},$$
$$X_{1/2}^\sharp=(X_0^\sharp,X_1^\sharp)_{1/2}=\{u\in H_{q'}^1(\Omega):u|_{\partial\Omega}=0\}$$
and
$X_{-1/2}=(X_{1/2}^\sharp)'$, see the Appendix. Moreover, the operator $A_{-1/2}:X_{1/2}\to X_{-1/2}$ is given by
$$\langle A_{-1/2}u,\phi\rangle=\int_\Omega\nabla u\cdot \nabla \phi\ dx$$
for all $(u,\phi)\in X_{1/2}\times X_{1/2}^\sharp$, which follows from integration by parts and the density of $X_1$ in $X_{1/2}$.

So in the week setting, we choose $X_0^{\sf w}=X_{-1/2}$, $X_1^{\sf w}=X_{1/2}$ and $A^{\sf w}=A_{-1/2}$ with domain $X_{1/2}$ to rewrite \eqref{eq:Ex1} as the semilinear evolution equation
\begin{equation}\label{eq:Ex1weak}
\partial_tu+A^{\sf w} u=F^{\sf w}(u),
\end{equation}
where $\langle F^{\sf w}(u)|\phi\rangle:=(F(u)|\phi)_{L_2(\Omega)}$. In the sequel let
\begin{equation*}
{_0 H}_q^{r}(\Omega):=
\left\{
\begin{aligned}
&\{u\in H_q^{r}(\Omega):u|_{\partial\Omega}=0\} && \text{for} && r\in (1/q,1],\\
&H_q^{r}(\Omega) && \text{for} && r\in [0,1/q),\\
\end{aligned}
\right.
\end{equation*}
and ${_0}H_q^{-r}(\Omega):=({_0}H_{q'}^{r}(\Omega))'$ if $r\in [0,1]\backslash\{1-1/q\}$. It follows readily from the reiteration property \eqref{eq:ReitProp} that 
$$X_\beta^{\sf w}=(X_0^{\sf w},X_1^{\sf w})_{\beta}={_0}H_q^{2\beta-1}(\Omega).$$
For $u\in X_\beta^{\sf w}$ and $\phi\in{_0}H_{q'}^1(\Omega)$, by H\"older's inequality, we therefore obtain
$$|\langle F(u)|\phi \rangle | \le |u|_{L_{r\kappa}(\Omega)}^\kappa|\phi|_{L_{r'}(\Omega)},$$
where $1/r+1/r'=1$. We choose $q,r,q'$ and $r'$ such that
$$H_{q}^{2\beta-1}(\Omega)\hookrightarrow L_{r\kappa}(\Omega)\quad\text{and}\quad H_{q'}^1(\Omega)\hookrightarrow L_{r'}(\Omega),$$
to be precise
$$2\beta-1-\frac{d}{q}= -\frac{d}{r\kappa}\quad\text{and}\quad 1-\frac{d}{q'}= -\frac{d}{r'}.$$
The last equality is equivalent to $1+d/q= d/r$. This yields
$$-\kappa\left(2\beta-1-\frac{d}{q}\right)= 1+\frac{d}{q}$$
or equivalently
$$\beta=\frac{1}{2}\left(1+d/q\right)\left(1-1/\kappa\right).$$
Note that the condition $\frac{1}{d}+\frac{1}{q}=\frac{1}{r}<1$ is equivalent to $q>\frac{d}{d-1}$.

Furthermore, the constraint $\beta<1$ leads to $q>d(\kappa-1)/(\kappa+1)$.
Under these assumptions we obtain the estimates
$$|F^{\sf w}(u)|_{X_0^{\sf w}}\le C|u|_{X_\beta^{\sf w}}^\kappa,$$
and
$$|F^{\sf w}(u)-F^{\sf w}(\bar{u})|_{X_0^{\sf w}}\le C(|u|_{X_\beta^{\sf w}}^{\kappa-1}+|\bar{u}|_{X_\beta^{\sf w}}^{\kappa-1})|u-\bar{u}|_{X_\beta^{\sf w}}.$$
From \textbf{(H3)} with $\rho_1=\kappa-1$ and $\beta_1=\beta$, we obtain the critical weight
$$\mu^{\sf w}_c=\frac{1}{p}+\frac{1}{2}\left(1+\frac{d}{q}\right)-\frac{1}{\kappa-1}=\frac{1}{p}+\frac{d}{2q}+\frac{\kappa-3}{2(\kappa-1)}.$$
which in turn yields the restriction
$$\frac{2}{p}+\frac{d}{q}\le \frac{\kappa+1}{(\kappa-1)}$$
as $\mu_c\le 1$. Note that the requirement $\mu_c>1/p$ leads to 
$$\frac{3-\kappa}{\kappa-1}<\frac{d}{q}$$
which is always satisfied provided $\kappa\ge 3$. We assume this in the sequel.

For $s\in (0,1)\backslash\{1/q\}$, we define the spaces
\begin{equation*}
{_0}B_{qp}^s(\Omega):=(X_0,X_{1/2})_{s,p}=
\left\{
\begin{aligned}
& \{u\in B_{qp}^{s}(\Omega):u|_{\partial\Omega}=0\} && \text{for} && s\in (1/q,1),\\
& B_{qp}^{s}(\Omega) && \text{for} && s\in (0,1/q)
\end{aligned}
\right.
\end{equation*}
and set ${_0}B_{qp}^{-s}(\Omega):=({_0}B_{q'p'}^s(\Omega))'$ for $s\in (0,1)\backslash\{1-1/q\}$. It follows from \eqref{re-int} that
$$(X_0^{\sf w},X_1^{\sf w})_{s,p}=(X_{-1/2},X_{1/2})_{s,p}={_0}B_{qp}^{2s-1}(\Omega),$$
for all $s\in (0,1)$, where, by reiteration 
$${_0}B_{qp}^0(\Omega):=(X_{-1/2},X_{1/2})_{1/2,p}
=(X_{-\alpha},X_{\alpha})_{1/2,p}=B_{qp}^0(\Omega)$$
for any $\alpha\in (0,1/2q)$. Therefore, the critical space in the weak setting is given by
$$X_{\gamma,\mu_c}^{\sf w}=(X_0^{\sf w},X_1^{\sf w})_{\mu^{\sf w}_c-1/p,p}={_0}B_{qp}^{d/q-2/(\kappa-1)}(\Omega).$$
Then we have the following result.
\begin{theorem} \label{thm:Ex1weak}
Let $\kappa\ge3$, $p\in (1,\infty)$ and let $q\in \left(d\cdot\max\{\frac{1}{d-1},\frac{\kappa-1}{\kappa+1}\},\infty\right)$ 
satisfy $2/p+d/q\le (\kappa+1)/(\kappa-1)$.

Then, for each {$u_0\in {_0B}^{d/q-2/(\kappa-1)}_{qp}(\Omega)$}, equation \eqref{eq:Ex1weak}
admits a unique solution
$$ u \in H^1_{p,\mu_c}((0,a);{_0}H_q^{-1}(\Omega))\cap L_{p,\mu_c}((0,a); {_0}H_q^1(\Omega)),$$
for some {$a>0$}, with critical weight 
$$\mu^{\sf w}_c = \frac{1}{p}+\frac{1}{2}\left(1+\frac{d}{q}\right)-\frac{1}{\kappa-1}.$$ 
The solution exists on a maximal interval {$(0,t_+(u_0))$} and depends continuously on {$u_0$}.
In addition,
$$ u \in C([0,t_+); {_0B}^{d/q-1/(\kappa-1)}_{qp}(\Omega))\cap C((0,t_+);{_0 B}^{1-2/p}_{qp}(\Omega)),
$$
i.e.\ the solutions regularize instantly if {$2/p +d/q<(\kappa+1)/(\kappa-1)$}.
\end{theorem}
\noindent
Equation \eqref{eq:Ex1} has been considered by many authors in the last four decades, 
see for instance \cite{CDW09, Fuj66, QuSo07,Wei81, Wei86}. 
To the best of our knowledge, the results in Theorems~\ref{thm:Ex1strong} and  \ref{thm:Ex1weak}  are new.

\medskip
\noindent
{\bf Example 2.} Let $\Omega$ be a bounded domain of class $C^{2}$ in $\RR^d$, and consider the Neumann problem
\begin{equation}\label{eq:ex2}
\begin{aligned}
\partial_t u -{\rm div}(a(u)\nabla u) &= |\nabla u|^\kappa && \mbox{in}\;\; \Omega, \\
\partial_\nu u&=0  && \mbox{on}\;\; \partial\Omega,\\
u&=u_0  && \mbox{in} \;\;\Omega, 
\end{aligned}
\end{equation}
where $\kappa>2$. Observe that this problem is a scaling invariant with respect to
$$u_\lambda(t,x)=\lambda^{-\frac{\kappa-2}{2(\kappa-1)}}u(\lambda t,\lambda^{\frac{1}{2}}x)$$
if e.g.~$a(u)\equiv a_0=const$. In the sequel, we assume that $a\in C^{1}(\R)$, $a(s)>0$ for all $s\in\R$, and
there exists $C>0$ such that
\begin{equation}\label{eq:a_cond}
|a'(s_1)-a'(s_2)|\le C|s_1-s_2|,\quad s_1,s_2\in\R.
\end{equation}
For sufficiently smooth $u$, this yields 
$${\rm div}(a(u)\nabla u)=a(u)\Delta u+a'(u)|\nabla u|^2,$$
hence we may rewrite $\eqref{eq:ex2}_1$ as
$$\partial_t u -a(u)\Delta u = |\nabla u|^\kappa+a'(u)|\nabla u|^2\quad \mbox{in } \Omega.$$
Let $X_0=L_q(\Omega)$, $1<q<\infty$, and
$$X_1=\{u\in H_q^2(\Omega):\partial_\nu u=0\ \text{on}\ \partial\Omega\}.$$
This in turn implies
\begin{equation*}
X_\beta=(X_0,X_1)_{\beta}=
\left\{
\begin{aligned}
& \{u\in H_q^{2\beta}(\Omega):\partial_\nu u=0\} &&\text{for} && \beta\in (1/2+1/2q,1],\\
& H_q^{2\beta}(\Omega) && \text{for} &&\beta\in [0,1/2+1/2q).
\end{aligned}
\right.
\end{equation*}
Define $F_\kappa,F_a:X_\beta\to X_0$ by
$$F_\kappa(u)=|\nabla u|^\kappa\quad\text{and}\quad F_a(u)=a'(u)|\nabla u|^2.$$
Note that if $H^{2\beta}_q(\Omega\hookrightarrow H_{q\kappa}^1(\Omega)$, i.e.
$$\beta=\frac{1}{2}+\frac{d}{2q}\left(1-\frac{1}{\kappa}\right),$$
the nonlinearities $F_\kappa$ and $F_a$ are well-defined, since $a'\in C(\R)$, $\kappa>2$ and $H_q^{2\beta}\hookrightarrow C(\overline{\Omega})$ for any $q\in (1,\infty)$. For $u_1,u_2\in X_\beta$, we obtain as in Example 1
$$|F_\kappa(u_1)-F_\kappa(u_2)|_{X_0}\le C(|u_1|_{X_\beta}^{\kappa-1}+|u_2|_{X_\beta}^{\kappa-1})|u_1-u_2|_{X_\beta}.$$
For this nonlinearity, the smallest possible value for $\mu$ may be computed from \textbf{(H3)} with $\rho_1=\kappa-1$ and $\beta_1=\beta$ to the result
$$\mu_c=\frac{1}{p}+\frac{\kappa\beta-1}{\kappa-1}=\frac{1}{p}+\frac{d}{2q}+\frac{\kappa-2}{2(\kappa-1)}.$$
Setting
\begin{equation*}
{_\nu}B_{qp}^{2s}(\Omega):=(X_0,X_1)_{s,p}=
\left\{
\begin{aligned}
& \{u\in B_{qp}^{2s}(\Omega):\partial_\nu u=0\} &&\text{for} &&s\in(1/2+1/2q,1),\\
& B_{qp}^{2s}(\Omega) && \text{for} && s\in (0,1/2+1/2q),
\end{aligned}
\right.
\end{equation*}
the critical space is
$$X_{\gamma,\mu_c}=(X_0,X_1)_{\mu_c-1/p,p}={_\nu B}_{qp}^{d/q+(\kappa-2)/(\kappa-1)}(\Omega),$$ 
which is embedded into $C(\overline{\Omega})$ for all $p,q\in (1,\infty)$, $\kappa>2$.

Concerning $F_a$, we write
$$F_a(u_1)-F_a(u_2)=a'(u_1)(|\nabla u_1|^2-|\nabla u_2|^2)+|\nabla u_2|^2(a'(u_1)-a'(u_2)),$$
for $u_1,u_2\in X_\beta$. It follows from H\"older's inequality and \eqref{eq:a_cond}
\begin{align*}
|a'(u_1)(|\nabla u_1|^2-|\nabla u_2|^2)|_{X_0}&\le C(1+|u_1|_{X_{\gamma,\mu_c}})(|\nabla u_1|_{L_{2q}}+|\nabla u_2|_{L_{2q}})|\nabla u_1-\nabla u_2|_{L_{2q}}\\
&\le C_\kappa(|u_1|_{X_{\gamma,\mu_c}})(1+| u_1|_{X_\beta}^{\kappa-1}+| u_2|_{X_\beta}^{\kappa-1})|u_1- u_2|_{X_\beta},
\end{align*}
since $\kappa-1>1$, where we also applied Young's inequality. So for this part we may set as before $\rho_2=\kappa-1$ and $\beta_2=\beta$. For the remaining part, we obtain
$$|\nabla u_2|^2(a'(u_1)-a'(u_2))|_{X_0}\le C|u_1-u_2|_{L_\infty}|u_2|_{H_{2q}^1}^2\le C|u_2|_{X_\beta}^2|u_1-u_2|_{X_{\gamma,\mu_c}}.$$
A straightforward calculation shows that there is strict inequality in \textbf{(H3)} with $\rho_3=2$ and $\beta_3=\mu_c-1/p$ if and only if $\beta<1$. 

In other words, this part of the nonlinearity $F_a$ is always subcritical and so $\mu_c$ defined above is the critical weight for the nonlinearity $F(u):=F_\kappa(u)+F_a(u)$. Note that the condition $\mu_c>1/p$ is always satisfied as $\kappa>2$. The restrictions
$\beta<1$ and  $\mu_c\le 1$ then lead to
\begin{equation*}
\begin{aligned}
& \frac{d}{q} < \frac{\kappa}{(\kappa-1)}\quad\text{and}\quad
&\frac{1}{p}+\frac{d}{2q}+\frac{\kappa-2}{2(\kappa-1)}\le 1\Leftrightarrow 
\frac{2}{p}+\frac{d}{q}\le \frac{\kappa}{(\kappa-1)},
\end{aligned}
\end{equation*}
respectively.

For $v\in X_{\gamma,\mu_c}$ we define an operator $A(v):X_1\to X_0$ by
$[A(v)u](x):=a(v(x))\Delta u(x)$, $x\in\Omega$. By compactness, there exists $a_0>0$ such that $a(v(x))\ge a_0>0$ for all $x\in\Omega$, since $x\mapsto a(v(x))$ is continuous. Furthermore,  it follows from \cite{DDHPV04} that $A_\#:=A(v)\in\mathcal{H}^\infty(X_0)$ with $\mathcal{H}^\infty$-angle $\phi_{A_\#}^\infty<\pi/2$. 

\begin{theorem} \label{thm:Ex2}
Let $a\in C^1(\R) $, $a(s)>0$ for all $s\in\R$ and assume \eqref{eq:a_cond}. Suppose that $\kappa>2$, $p\in (1,\infty)$, and 
$2/p+d/q\le \kappa/(\kappa-1)$.

Then, for each $u_0\in {_\nu B}^{d/q+(\kappa-2)/(\kappa-1)}_{qp}(\Omega)$, problem \eqref{eq:ex2}
admits a unique solution
$$ u \in H^1_{p,\mu_c}((0,a);L_q(\Omega))\cap L_{p,\mu_c}((0,a); H_q^2(\Omega)),$$
for some $a>0$, with critical weight $\mu_c = 1/p+ d/2q-(\kappa-2)/2(\kappa-1)$. 
The solution exists on a maximal interval $(0,t_+(u_0))$ and depends continuously on $u_0$.
In addition,
$$ u \in C([0,t_+); {_\nu B}^{d/q+(\kappa-2)/(\kappa-1)}_{qp}(\Omega))\cap C((0,t_+);{_\nu B}^{2(1-1/p)}_{qp}(\Omega)),
$$
i.e.\ the solutions regularize instantly if $2/p +d/q<\kappa/(\kappa-1)$.
\end{theorem}
\noindent
{The results contained in Theorem~\ref{thm:Ex2} seem to be new.
We refer to the monograph \cite{QuSo07} for additional results concerning equation \eqref{eq:ex2}.}\\

\medskip

\noindent
{\bf Example 3.} 
Let $\Omega\subset\R^d$ be a bounded domain with boundary $\partial\Omega\in C^{4-}$. 
Consider the \emph{Cahn-Hilliard equation}
\begin{equation}
\label{eq:CH}
\begin{aligned}
\partial_t u-\Delta v&=0 && \text{in} \;\;\Omega,\\
v+\Delta u-\Phi'(u)&=0 && \text{in}\;\; \Omega,\\
\partial_\nu u=\partial_\nu v&=0 &&\text{on}\;\; \partial\Omega,\\
u(0)&=u_0 &&\text{in}\;\;\Omega.
\end{aligned}
\end{equation}
Here $u$ is an order parameter, $v$ is the chemical potential and $\Phi$ denotes the physical potential, which is often assumed to be of double-well type, i.e.~ $\Phi(s)=(s^2-1)^2$ for $s\in\R$. Note that in view of the homogeneous Neumann boundary conditions, the elliptic-parabolic problem \eqref{eq:CH} is equivalent to the purely parabolic problem
\begin{equation}
\label{eq:CH2}
\begin{aligned}
\partial_t u+\Delta^2 u-\Delta\Phi'(u)&=0 &&\text{in}\;\; \Omega,\\
\partial_\nu u=\partial_\nu \Delta u&=0 &&\text{on}\;\; \partial\Omega,\\
u(0)&=u_0 && \text{in}\;\;\Omega.
\end{aligned}
\end{equation}
Let $X_0:=L_q(\Omega)$, $1<q<\infty$,
$$X_1:=\{u\in H_q^4(\Omega):\partial_\nu u=\partial_\nu\Delta u=0\}$$
and define an operator $A_0:X_1\to X_0$ by $A_0u:=\Delta^2 u$. 
By \cite{DDHPV04}, $A_0\in\mathcal{H}^\infty(X_0)$ with $\mathcal{H}^\infty$-angle $\phi_{A_0}^\infty<\pi/2$. Let $(X_\alpha,A_\alpha)$, $\alpha\in\R$, denote the interpolation-extrapolation scale with respect to the complex interpolation functor (see the Appendix). Then $X_{-1/2}=(X_{1/2}^\sharp)'$ and
$$X_{1/2}=\{u\in H_{q}^2(\Omega):\partial_\nu u=0\},$$ 
as well as
$$X_{1/2}^\sharp=\{u\in H_{q'}^2(\Omega):\partial_\nu u=0\}.$$ 
The operator $A_{-1/2}:X_{1/2}\to X_{-1/2}$ may be represented as
$$\langle A_{-1/2}u,v\rangle=\int_\Omega \Delta u\,\Delta v\ dx$$
for all $(u,v)\in X_{1/2}\times X_{1/2}^\sharp$, which follows from integration by parts and density of $X_1$ in $X_{1/2}$. We consider the weak formulation 
\begin{equation}
\label{eq:CHweak}
\partial_tu+A^{\sf w} u=F^{\sf w}(u),\ t>0,\quad u(0)=u_0,
\end{equation}
of \eqref{eq:CH} in $X_0^{\sf w}:=X_{-1/2}$, where $A^{\sf w}=A_{-1/2}$ with domain $X_1^{\sf w}:=X_{1/2}$ and 
$$\langle F^{\sf w}(u),v\rangle:=(\Phi'(u)|\Delta v)_{L_2(\Omega)}.$$ 
In the sequel, we assume $\Phi\in C^1(\R)$ and there are constants $C>0$ and $\kappa>0$ such that 
\begin{equation}
\label{eq:condPhi}
|\Phi'(s)-\Phi'(\bar{s})|\le C(1+|s|^\kappa+|\bar{s}|^\kappa)|s-\bar{s}|,\quad s,\bar s\in\R.
\end{equation}
H\"older's inequality in combination with \eqref{eq:condPhi} yields
$$|F^{\sf w}(u)|_{X_{0}^{\sf w}}\le |\Phi'(u)|_{L_q}\le C(1+|u|_{L_{(\kappa+1)q}}^{\kappa+1}),$$
for all $u\in X_\beta^{\sf w}=(X_0^{\sf w},X_1^{\sf w})_\beta={_\nu H}_q^{4\beta-2}(\Omega)$, $\beta\in (0,1)$ (see the Appendix for details). 
Here 
\begin{equation*}
{_\nu H}_q^{r}(\Omega):=
\left\{
\begin{aligned}
& \{u\in H_q^{r}(\Omega):\partial_\nu u=0\} && \text{for} && r\in (1+1/q,2],\\
& H_q^{r}(\Omega) &&\text{for} && r\in [0,1+1/q),\\
\end{aligned}
\right.
\end{equation*}

and ${_\nu}H_q^{-r}(\Omega):=({_\nu}H_{q'}^{r}(\Omega))'$ if $r\in [0,2]\backslash\{2-1/q\}$.

Thus, $F^{\sf w}:X_\beta^{\sf w}\to X_0^{\sf w}$ is well-defined, provided $H_q^{4\beta-2}\hookrightarrow L_{(\kappa+1)q}$, hence
$$4\beta-2-\frac{d}{q}=-\frac{d}{(\kappa+1)q}\Leftrightarrow \beta=\frac{1}{2}+\frac{d\kappa}{4q(\kappa+1)}.$$
The condition $\beta<1$ is then equivalent to $q>\frac{d\kappa}{2(\kappa+1)}$
and the estimate
$$|F^{\sf w}(u)-F^{\sf w}(\bar{u})|_{X_0^{\sf w}}\le C(1+|u|_{X_\beta^{\sf w}}^{\kappa}+|\bar{u}|_{X_\beta^{\sf w}}^{\kappa})|u-\bar{u}|_{X_\beta^{\sf w}}.$$
holds for all $u,\bar{u}\in X_\beta^{\sf w}$. From \textbf{(H3)} with $\rho_1=\kappa$ and $\beta_1=\beta$, we obtain the critical weight
$$\mu_c=\frac{1}{p}+\frac{d}{4q}+\frac{\kappa-1}{2\kappa}.$$
Hence $\mu_c\le 1$ iff
$$\frac{1}{p}+\frac{d}{4q}\le \frac{\kappa+1}{2\kappa},$$
and
in case $\kappa\in (0,1)$  it is $\mu_c>1/p$ if and only if
$$q<\frac{d\kappa}{2(1-\kappa)}.$$
As $q>1$, this implies the lower bound $\kappa>2/(d+2)$.
The critical  space is then given by
$$X_{\gamma,\mu_c}^{\sf w}={_\nu}B_{qp}^{4(\mu_c-1/p)-2}(\Omega)={_\nu}B_{qp}^{d/q-2/\kappa}(\Omega),$$
where
\begin{equation*}
{_\nu B}_{qp}^{r}(\Omega):=
\left\{
\begin{aligned}
& \{u\in B_{qp}^{r}(\Omega):\partial_\nu u=0\} &&\text{for} && r\in (1+1/q,2],\\
& B_{qp}^{r}(\Omega) &&\text{for} && r\in [0,1+1/q),\\
\end{aligned}
\right.
\end{equation*}
and ${_\nu}B_{qp}^{-r}(\Omega):=({_\nu}B_{q'p'}^{r}(\Omega))'$ if $r\in [0,2]\backslash\{2-1/q\}$.
In case of the double-well potential $\Phi(s)=(s^2-1)^2$ we may set $\kappa=2$ in \eqref{eq:condPhi}, hence the critical space for this case reads $${_\nu}B_{qp}^{d/q-1}(\Omega).$$
\begin{theorem} \label{thm:Ex3}
Let $\Phi\in C^1(\R) $ and assume \eqref{eq:condPhi}. Suppose that $p\in (1,\infty)$ and $2/p+d/2q\le (\kappa+1)/\kappa$, 
where $q\in (1,\infty)$, $\kappa>2/(d+2)$ and 
$q<d\kappa/(2(1-\kappa))$ in case $\kappa\in (2/(d+2),1)$. 
Then, for each {$u_0\in {_\nu B}^{d/q-\kappa/2}_{qp}(\Omega)$}, problem \eqref{eq:CHweak}
admits a unique solution
$$ u \in H^1_{p,\mu_c}((0,a);{_\nu}H_q^{-2}(\Omega))\cap L_{p,\mu_c}((0,a); {_\nu}H_q^2(\Omega)),$$
for some {$a>0$}, with critical weight {$\mu_c = 1/p+ d/4q+(\kappa-1)/2\kappa$}. The solution exists on a maximal interval {$(0,t_+(u_0))$} and depends continuously on {$u_0$}.
In addition,
$$ u \in C([0,t_+); {_\nu B}^{d/q-\kappa/2}_{qp}(\Omega))\cap C((0,t_+);{_\nu B}^{2-4/p}_{qp}(\Omega)),
$$
i.e.\ the solutions regularize instantly if {$2/p +d/2q<(\kappa+1)/\kappa$}.
\end{theorem}
 \noindent
{The Cahn-Hilliard equation has been proposed in the pioneering work \cite{CaHi}, to model the separation of phases of a binary fluid. It has been subject of intensive research during the last decades, see for instance
\cite{AbWi,ElZhe,HR99} and the references therein.
So far, there seem to be no results on critical spaces for the Cahn-Hilliard equation.}

%

\section{Vorticity Equation}
Let $\Omega$ be a bounded, simply connected domain in $\RR^3$ with boundary $\Sigma:=\partial\Omega$ of class $C^{3}$. We consider the Navier-Stokes equation with boundary conditions of Navier type.
\begin{equation}
\label{NS}
\begin{aligned}
\partial_t u -\upmu\, \Delta u +u\cdot\nabla u +\nabla \pi&=0 && \mbox{in}\;\; \Omega,\\
{\rm div}\, u &=0 &&\mbox{in}\;\; \Omega,\\
u\cdot\nu =0,\; 2\upmu\, {\sf P}_\Sigma D(u) \nu +\alpha {\sf P}_\Sigma u &=0 && \mbox{on}\;\; \Sigma,\\
u(0)&=u_0 && \mbox{in}\;\; \Omega.
\end{aligned}
\end{equation}
Here $\upmu>0$ and $\alpha\geq0$ are constants, ${\sf P}_\Sigma = I-\nu\otimes\nu$
is the orthogonal projection onto the tangent bundle $T\Sigma$,
$D(u) = (\nabla u +[\nabla u]^{\sf T})/2$ denotes the symmetric velocity gradient and
$R(u)= (\nabla u -[\nabla u]^{\sf T})/2$ its asymmetric part, for future reference. 
The parameter $\alpha\ge 0$ takes friction on the boundary $\Sigma$ into consideration. 
If $\alpha=0$, we are in the case of  
pure-slip boundary conditions, whereas $\alpha>0$ corresponds to the case 
of partial-slip.  
We want to study this problem in terms of its vorticity and its stream function. By proper scaling we may assume $\upmu=1$.

\noindent
\subsection{The Navier Condition.}  We first reformulate the Navier boundary condition in a way that is more convenient
for our analysis. 
For this purpose, we make use of the splitting
\begin{equation*}
w=(I- \nu\otimes\nu)w + (w\cdot\nu)\nu=: w_\parallel +w_\nu\nu,
\end{equation*}
where $w_\parallel$ and $w_\nu$ denote the tangential and the normal part, respectively, of a vector field $w$
defined on $\Sigma$.
It is important to note that we can extend the unit normal 
$\nu$, defined on $\Sigma$, to a tubular neighborhood $U$ of $\Sigma$ according to 
$\tilde\nu(x) = \nu(\Pi_\Sigma(x))$, where $\Pi_\Sigma$ denotes the metric projection onto $\Sigma$; 
see for instance Pr\"uss and Simonett \cite[Chapter 2]{PrSi16} for more details.
We can then also extend the above decomposition of $w$ to the tubular neighborhood $U$.
In the following we always assume that
a given vector field $w$ defined on $\bar\Omega \cap U$ is decomposed in a tangential and normal component according to
\begin{equation}
\label{splitting-extended}
w=(I- \tilde\nu\otimes\tilde\nu)w + (w\cdot\tilde\nu)\tilde\nu=: w_\parallel +w_\nu\tilde\nu.
\end{equation}
In order to not overburden the notation, we will drop the tilde in the sequel.
Note that with this convention we have
\begin{equation}
\label{derivatives}
\partial_\nu \nu=0, \quad \partial_\nu w_\parallel\cdot \nu=0 \quad\mbox{on}\;\; \Sigma,
\end{equation}
for any vector field $w$ defined on $\bar\Omega \cap U$.
An easy computation then yields
\begin{equation*}
\begin{aligned}
2{\sf P}_\Sigma D(u) \nu &= \partial_\nu u_\parallel +{\sf L}_\Sigma u_\parallel +\nabla_\Sigma u_\nu &\text{on}\;\;\Sigma ,\\
2{\sf P}_\Sigma R(u) \nu &= -\partial_\nu u_\parallel +{\sf L}_\Sigma u_\parallel +\nabla_\Sigma u_\nu &\text{on}\;\;\Sigma , 
\end{aligned}
\end{equation*}
where
${\sf L}_\Sigma=-\nabla_\Sigma \nu$ is the Weingarten tensor and 
$\nabla_\Sigma$ denotes the surface gradient on $\Sigma$, see also \cite[Section 5.4]{PrWi17a}.
For dimension 3, we obtain
$$ 2{\sf P}_\Sigma R(u) \nu= \nu\times {\rm rot}\, u. $$
As $u_\nu=0$ by the first boundary condition, with 
${\sf B}_\Sigma = -2 {\sf L}_\Sigma -\alpha {\sf P}_\Sigma$ the second implies 
\begin{align*}
0&= 2{\sf P}_\Sigma D(u) \nu +\alpha {\sf P}_\Sigma u 
= \partial_\nu u_\parallel +{\sf L}_\Sigma u_\parallel +\alpha u_\parallel\\
&=  \partial_\nu u_\parallel -{\sf L}_\Sigma u_\parallel -{\sf B}_\Sigma u_\parallel
 = {\rm rot}\, u \times \nu -  {\sf B}_\Sigma u_\parallel,
\end{align*}
hence the Navier boundary conditions are equivalent to
\begin{equation}\label{nbc}
u\cdot\nu=0,\quad {\rm rot}\, u \times \nu = {\sf B}_\Sigma u\quad \mbox{on } \Sigma.
\end{equation}
Note that this is a lower order perturbation of the so-called {\em perfect slip} boundary conditions
$$  u\cdot\nu=0,\quad {\rm rot}\,u \times \nu =0.$$

\noindent
 \subsection{The Stream Function.} As ${\rm div}\, u=0$ in $\Omega$, $u\cdot\nu=0$ on $\Sigma=\partial\Omega$, and $\Omega$ is simply connected by assumption, there is a unique solution $v$ of the problem
\begin{equation}
\begin{aligned}
\label{SF}
{\rm rot}\, v&=u && \mbox{in}\;\; \Omega,\\
{\rm div}\, v&=0 && \mbox{in}\;\; \Omega,\\
v_\parallel &=0  && \mbox{on}\;\; \Sigma.
\end{aligned}
\end{equation}
We call $v$ the {\em stream function}, below.

\medskip\noindent
{\bf (i)}
We shall first show uniqueness. 
Suppose $u=0$ and $v$ is a solution of \eqref{SF}.
As ${\rm rot}\,v=0$ and $\Omega$ is simply connected, there exists a potential function $\phi$,
i.e. we have $v=\nabla \phi$. 
From the second line follows $\Delta\phi=0$ in $\Omega$, while 
the boundary condition $v_\parallel=0$ implies $\nabla_\Sigma \phi=0$ on $\Sigma$.
Hence, there is a constant $c$ such that $\phi\equiv c$ on $\Sigma$ and the function $\phi-c$ then solves the elliptic problem\
$$\Delta (\phi-c)=0\;\;\text{in}\;\;\Omega, 
\quad \phi-c=0 \;\;\text{on}\;\;\Sigma.$$ 
Therefore, $\phi$ is constant in all of $\Omega$ and $v=\nabla\phi=0$.

\medskip
\noindent
{\bf (ii)} 
In order to show existence, we consider the elliptic problem
\begin{equation}
\begin{aligned}
\label{LP}
-\Delta v &= {\rm rot}\, u && \mbox{in}\;\; \Omega,\\
v_\parallel &=0            && \mbox{on}\;\; \Sigma,\\
\partial_\nu v_\nu - \kappa_\Sigma v_\nu &=0     && \mbox{on}\;\; \Sigma,
\end{aligned}
\end{equation}
where $\kappa_\Sigma$ is the mean curvature (more precisely, the sum of the principal curvatures) of $\Sigma$.
For later use we record the important relationship
\begin{equation}
\label{div-boundary}
{\rm div}\,v = {\rm div}_\Sigma v_\parallel + \partial_\nu v_\nu -\kappa_\Sigma v_\nu\;\;\text{on}\;\;\Sigma.
\end{equation}
Hence the boundary conditions of \eqref{LP} imply ${\rm div}\,v=0$ on $\Sigma$.

Problem~\eqref{LP} is uniquely solvable, with $v=A^{-1}{\rm rot}\, u$,
where the operator $A$ is defined in Section 4.4 below.

\medskip
\noindent
{\bf (iii)} It remains to show that ${\rm div}\,v=0$ and ${\rm rot}\, v=u$ in $\Omega$.
The first assertion readily follows from the observation that
the solution $v$ of~\eqref{LP} satisfies
\begin{equation*}
\begin{aligned}
\Delta\, {\rm div}\,v &=0 &&\mbox{in}\;\; \Omega,\\
        {\rm div}\,v &=0 &&\mbox{on}\;\; \Sigma,   
\end{aligned}
\end{equation*}
which only admits the trivial solution. 
Hence, ${\rm rot}\,({\rm rot}\,  v -u)=0$ in $\Omega$,
and by simple connectedness of $\Omega$ this yields ${(\rm rot}\,v -u )=\nabla\phi$ for some harmonic function $\phi$.
We claim that
$({\rm rot}\,v -u)\cdot\nu=0$ on $\Sigma$. 
By assumption, $u\cdot\nu=0$. Hence
\begin{equation*}
({\rm rot}\,v-u)\cdot\nu 
=\big((\nabla_\Sigma +\nu\partial_\nu)\times (v_\parallel + v_\nu\nu)\big)\cdot\nu
=\big((\nabla_\Sigma v_\nu-\partial_\nu v_\parallel)\times \nu\big)\cdot\nu =0
\;\;\mbox{on}\;\;\Sigma.
\end{equation*}
Here we used  the properties that
$v_\parallel=0$, $\partial_\nu\nu=0$, and  $\nabla_\Sigma\times\nu=0$ on $\Sigma$.
The latter assertion can be verified by means of local coordinates, see for instance Section 2.1 in \cite{PrSi16},
\begin{equation*}
\nabla_\Sigma\times\nu
=\tau^j\partial_j\times\nu =\tau^j \times \partial_j\nu
=\tau^j \times l_{jk}\tau^k =\tau^1\times l_{12}\tau^2 + \tau^2\times l_{21}\tau^1=0, 
\end{equation*} 
as $l_{12}=l_{21}$.
Noting that $\phi$ solves the elliptic problem
\begin{equation*}
\Delta \phi =0 \;\;\mbox{in}\;\; \Omega, \quad
     \partial_\nu\phi =({\rm rot}\,v -u )\cdot\nu=0\;\;\mbox{on}\;\; \Sigma,
\end{equation*}
we conclude that $\phi$ is constant on $\Omega$,
and hence ${(\rm rot}\,v -u )=\nabla\phi=0$ in $\Omega$, i.e. ${\rm rot}\,v = u$.

\medskip\noindent
For $q\in (1,\infty)$ and  $s\in [-1,1]\setminus\{1/q, 1/q-1\}$ 
one can show that  
$$ 
A^{-1}{\rm rot}:  {_{\sf N} H}^{s}_{q}(\Omega) \to {_\parallel H}^{s+1}_q(\Omega),
$$
is linear and bounded,
where ${_{\sf N}H}^s_q(\Omega)$ is the complex interpolation-extrapolation scale associated to problem~\eqref{NS}, and where
the spaces ${_\parallel H}^s_q(\Omega)$ are defined in Section~4.4 below.
 
\noindent
\subsection{ The Vorticity Equation.}
We define the {\em vorticity} by means of $w:={\rm rot}\, u$.  As a consequence of 
the above considerations, we then have
$$ u ={\rm rot}\, v  ={\rm rot}\, A^{-1} {\rm rot}\, u ={\rm rot}\, A^{-1}w =:  L_0 w.$$
Observe that $ L_0$ is an operator of order $-1$.
This way $u$ is determined uniquely by the stream function $v$, or equivalently by the vorticity $w$.  
The latter property is usually referred to as the Biot-Savart law.
Note that
\begin{equation}
\begin{aligned}
\label{vortex-relation}
{\rm rot}\,({\rm div}\,(u\otimes u))
&={\rm rot}\,(u\cdot\nabla u) 
= (u\cdot\nabla)\, {\rm rot}\,u -({\rm rot}\, u\cdot\nabla) u\\
&= u\cdot\nabla w-w\cdot \nabla u
={\rm div}\,(u\otimes w)-{\rm div}\,(w\otimes u),
\end{aligned}
\end{equation}
as ${\rm div}\,u = {\rm div}\,w=0$.
The vorticity equation now reads
\begin{equation}
\begin{aligned}
\label{vorteq}
\partial_t w + u\cdot\nabla w -w\cdot\nabla u-\Delta w &=0 && \mbox{in}\;\; \Omega,\\
u &=  L_0 w && \mbox{in } \Omega, \\
w\times \nu =B_\Sigma u &=0 && \mbox{on}\;\; \Sigma,\\ 
 {\rm div}_\Sigma\, w_\parallel  + \partial_\nu w_\nu  -\kappa_\Sigma w_\nu &=0 && \mbox{on}\;\; \Sigma,\\
w(0)&=w_0 && \mbox{in}\;\; \Omega.
\end{aligned}
\end{equation}
Here we note that the second boundary condition ensures ${\rm div}\, w=0$, as soon as ${\rm div}\, w_0=0$, which is natural as $w_0={\rm rot }\, u_0$.
In fact, for $\phi={\rm div}\, w$ we then obtain, formally at least,
\begin{equation*}
\begin{aligned}
\partial_t \phi -\Delta \phi&=0 && \mbox{in}\;\; \Omega, \\
 \phi &=0 && \mbox{on}\;\; \Sigma,\\
\phi(0)&=0 && \mbox{in}\;\; \Omega,
\end{aligned}
\end{equation*}
hence ${\rm div}\, w =\phi=0$.
We also observe that with
$$ \nu\times(w\times\nu) = w(\nu\cdot\nu) -\nu(w\cdot\nu) = w_\parallel,$$
the first boundary condition for $w$ can be rephrased as
\begin{equation}
\label{w-parallel}
 w_\parallel =\nu\times {\sf B}_\Sigma u =:  L_1w.
\end{equation}
Observe that $L_1= \nu\times {\sf B}_\Sigma L_0$ is an operator of order $-1$, hence a lower order perturbation.
We also recall that we have
$$u\cdot\nabla u = {\rm div}(u\otimes u),$$
as ${\rm div}\, u=0$. This will be useful for the very weak formulation below.

\medskip

\noindent
\subsection{ The Scale of the Principal Operator $A$.}
We define the principal operator $A$ in $L_q(\Omega)^3$ by means of
\begin{equation} \label{princop}
Aw=-\Delta w, \quad w\in {\sf D}(A) =\{ w\in H^2_q(\Omega)^3:
\, w_\parallel =0,\;\; \partial_\nu w_\nu-\kappa_\Sigma w_\nu=0\}.
\end{equation}
The operator $A$ has some beautiful properties.
Firstly, it 
admits an $\cH^\infty$-calculus with $\cH^\infty$-angle $\phi_A^\infty=0$ in $L_q(\Omega)^3$. 
Next we note that $A$ is positive definite in $L_2(\Omega)^3$. In order to see this, we employ the relation
$- \Delta  w = {\rm rot}\,{\rm rot}\, w -  \nabla {\rm div}\, w$ and integrate by parts. This yields
\begin{equation}
\label{Aww}
\begin{aligned}
-(\Delta w|w)_{\Omega} &= |{\rm rot}\, w|^2_{\Omega} + (\nu\times {\rm rot}\, w|w)_{\Sigma} 
 + |{\rm div}\, w|^2_{\Omega} - ({\rm div}\,w|w_\nu)_{\Sigma} \\ 
&=|{\rm rot}\, w|^2_{\Omega} + |{\rm div}\, w|^2_{\Omega},
\end{aligned}
\end{equation}
where $(u|v)_\Omega:=(u|v)_{L_2(\Omega)}$ and $(u|v)_\Sigma:=(u|v)_{L_2(\Sigma)}$.   
Here we used  $w_\parallel=0$,  and \eqref{div-boundary}. 
This shows that $A$ is positive semi-definite.
Suppose $Aw=0.$ Then \eqref{Aww} implies ${\rm rot}\,w=0$ and ${\rm div}\,w=0$ in $\Omega$.
By uniqueness of problem~\eqref{SF}, see Step (i) in Section 4.2, $w=0$. 
Hence $A$ is injective in $L_2(\Omega)^3$.
As $A$ has compact resolvent in $L_2(\Omega)^3$, its spectrum consists 
of eigenvalues of finite multiplicity.
Therefore, $0$ lies in the resolvent set of $A$ in $L_2(\Omega)^3$.
Since the spectrum of $A$ is independent of $q\in (1,\infty)$, we 
conclude that $A$ is invertible in $L_q(\Omega)^3$ for all $q\in (1,\infty)$.
In summary, $A$ is invertible, and the spectrum of $A$ in $L_q(\Omega)^3$ consists only of eigenvalues of 
finite multiplicity which are all positive.

The pair $(L_q(\Omega), A)$ generates an interpolation-extrapolation scale, see the Appendix, and
we have explicit expressions for the extrapolation-interpolation spaces, i.e. we have for 
${_\parallel H}^{2s}_q(\Omega):=(L_q(\Omega),{\sf D}(A))_s$

\begin{equation*}
{_\parallel H}^{2s}_q(\Omega)=
\left\{
\begin{aligned}
& \{w\in H^{2s}_q(\Omega)^3:\, w_\parallel=0,\, \partial_\nu w_\nu -{\kappa_\Sigma w_\nu} =0\}, && \!\!\! s\in (1/2+1/2q,1),\\
& \{w\in H^{2s}_q(\Omega)^3:\, w_\parallel=0\}, && \!\!\! s\in (1/2q,1/2+1/2q),\\
& H^{2s}_q(\Omega)^3, && \!\!\! s\in (0,1/2q),
\end{aligned}
\right.
\end{equation*}
and for ${_\parallel B}^{2s}_{qp}(\Omega):=(L_q,{\sf D}(A))_{s,p}$,
\begin{equation*}
{_\parallel B}^{2s}_{qp}(\Omega)=
\left\{
\begin{aligned}
& \{w\in B^{2s}_{qp}(\Omega)^3:\,w_\parallel=0,\, \partial_\nu w_\nu -{\kappa_\Sigma w_\nu} =0\}, && \!\!\!  s\in (1/2+1/2q,1),\\
& \{w\in B^{2s}_{qp}(\Omega)^3:\, w_\parallel=0\}, && \!\!\! s\in (1/2q,1/2+1/2q),\\
& B^{2s}_{qp}(\Omega)^3,  && \!\!\! s\in (0,1/2q).
\end{aligned}
\right.
\end{equation*}
Moreover,
\begin{equation*}
{_\parallel H}^{-2s}_q(\Omega):=\big({_\parallel H}^{2s}_{q^\prime}(\Omega)\big)^\prime,\quad
{_\parallel B}^{-2s}_{qp}(\Omega) := \big({_\parallel B}^{2s}_{q^\prime p^\prime}(\Omega)\big)^\prime, 
\end{equation*}
for $s\in [0,1]\setminus\{1/2-1/2q, 1-1/2q\}.$

\medskip
Moreover, $A$ commutes with the {\em Weyl projection} $\PP_W$ defined by
$w= \PP_W w+\nabla\varphi$, where
$$\Delta \varphi ={\rm div}\, w \quad \mbox{in } \Omega,\quad \varphi=0 \quad \mbox{on } \Sigma.$$
Therefore, its restriction $A_0$ to $X_0 :=\PP_W L_q(\Omega)^3=: {_\parallel L}_{q,\sigma}(\Omega)$  with domain
\begin{equation*}
\begin{aligned}
&X_1:={\sf D}(A_0)= \PP_W{\sf D}(A), \\
&X_1=\{ w\in H^2_q(\Omega)^3: {\rm div}\, w =0 \; \mbox{in}\; \Omega, w_\parallel = 0\, \mbox{on}\; \Sigma\}
=: {_\parallel H}^2_{q,\sigma}(\Omega), 
\end{aligned}
\end{equation*}
has the same properties as $A$.
We note on the go that the conditions ${\rm div}\,w=0$ and $w_\parallel =0$ imply 
{$\partial_\nu w_\nu - \kappa_\Sigma w_\nu=0$.}

Hence,  
the pair $(X_0,A_0)$ generates the complex interpolation scale $(X_\alpha,A_\alpha)$, $\alpha\in \RR$,
see the Appendix.
Here we are particularly interested in the cases $\alpha=-1/2$ for the weak formulation 
and $\alpha= -1$ for the very weak setting.  Observe that all these spaces are of class $U\!M\!D$, and all these operators 
admit an $\cH^\infty$-calculus with $\cH^\infty$-angle $\phi_A^\infty=0$.
The corresponding complex interpolation spaces are given by
$$ X_\alpha = {\sf D}(A_0^\alpha)=\PP_W{\sf D}(A^\alpha),\quad D_{A_0}(\alpha,p) = \PP_W D_A(\alpha,p),\quad \mbox{for} \; \alpha>0,$$
and
$$ X_{\alpha}= \big({\sf D}([A_0^\#]^{-\alpha})\big)^\prime, 
\quad D_{A_0}(\alpha,p)= \big(D_{A^\#_0}(-\alpha,p^\prime)\big)^\prime,\quad \mbox{for} \; \alpha<0.$$
Here $A_0^\#$ means $A_0$ considered in $X_0^\sharp =L_{q^\prime}(\Omega)$, i.e.\ in the dual scale.
In the sequel, we set
$$ {_\parallel H}^s_{q,\sigma}(\Omega):= X_{s/2} 
\quad \mbox{and}\quad {_\parallel B}^s_{qp,\sigma}(\Omega):= D_{A_0}(s/2,p):=(X_0,{\sf D}(A_0))_{s/2,p}.$$

\medskip

\noindent
\subsection{Very Weak Formulation.}
In the very weak formulation we define
$$X_0^{\sf vw}= X_{-1}= {_\parallel H}^{-2}_{q,\sigma}(\Omega), \quad X_1^{\sf  vw}= X_{0}=:{_\parallel L}_{q,\sigma}(\Omega).$$
Then for  $\phi \in {_\parallel H}^2_{q^\prime,\sigma}(\Omega)$ we obtain with two integrations by parts
\begin{align*}
0&= (\partial_t w -\Delta w +{\rm rot}\,(u\cdot\nabla u) |\phi)_\Omega\\
&= (\partial_t w|\phi)_\Omega -(w|\Delta\phi)_\Omega - (u\otimes u:{\nabla\rm rot}\, \phi)_\Omega
 -(\partial_\nu w|\phi)_\Sigma +(w|\partial_\nu\phi)_\Sigma\\
 &= \langle\partial_t w + A_{-1} w|\phi\rangle - \langle (B^{\sf vw} w+F^{\sf vw}(w)|\phi\rangle,
\end{align*}
with
$$ \langle F^{\sf vw}(w)|\phi\rangle= (u\otimes u|\nabla{\rm rot}\, \phi)_\Omega,\quad 
\langle B^{\sf vw} w|\phi\rangle =
(\nu\times{\sf B}_\Sigma u|\partial_\nu \phi_\parallel -\nabla_\Sigma \phi_\nu)_\Sigma.$$
Here the expression for $B^{\sf vw}$ was derived as follows.
From \eqref{derivatives}, \eqref{vorteq},  
the surface divergence theorem, and $\phi_\parallel=0$,  {$\partial_\nu\phi_\nu-\kappa_\Sigma\phi_\nu =0$} on $\Sigma$
we obtain
\begin{equation*}
\begin{aligned}
(\partial_\nu w|\phi)_\Sigma 
&=(\partial_\nu w_\parallel +(\partial_\nu w_\nu)\nu|\phi_\nu\nu)_\Sigma
=((\partial_\nu w_\nu)\nu|\phi_\nu\nu)_\Sigma \\
&= (-{\rm div}_\Sigma w_\parallel {+\kappa_\Sigma w_\nu} |\phi_\nu )_\Sigma 
=(w_\parallel |\nabla_\Sigma \phi_\nu)_\Sigma {+ ( \kappa_\Sigma w_\nu |\phi_\nu )_\Sigma }
 \end{aligned}
\end{equation*}
and 
\begin{equation*}
(w|\partial_\nu\phi)_\Sigma 
= {(w_\parallel | \partial_{\nu}\phi_\parallel)_\Sigma  + (w_\nu |\kappa_\Sigma \phi_\nu)_\Sigma.} 
\end{equation*}
 This shows that the very weak formulation of the vorticity equation reads
\begin{equation}\label{vorteq-vw}
\partial_t w + A_{-1} w = B^{\sf vw} w + F^{\sf vw}(w),\; t>0,\quad w(0)=w_0.
\end{equation}
Here $B^{\sf vw}$ is a linear lower order perturbation and $F^{\sf vw}$ is bilinear.

To show that $B^{\sf vw}$ is lower order, we estimate as follows.
\begin{align*}
 |\langle B^{\sf vw}w|\phi\rangle|&\leq |\nu \times {\sf B}_\Sigma u|_{L_q(\Sigma)}|\partial_\nu \phi_\parallel -\nabla_\Sigma \phi_\nu|_{L_{q^\prime}(\Sigma)}\\
 &\leq |u|_{H^{1/q+\varepsilon}_q(\Omega)} |\phi|_{H^{1+1/q^\prime+\varepsilon}_{q^\prime}(\Omega)}\\
& \leq |w|_{{_\parallel H}^{1/q+\varepsilon -1}_q(\Omega)}|\phi|_{H^{2-1/q+\varepsilon}_{q^\prime}(\Omega)}.
\end{align*}
Here we have employed 
the mapping properties of ${\rm rot}$ and $A^{-1}$ for $u={\rm rot}\, A^{-1}w$. 
This shows that
$$B^{\sf vw} : {_\parallel H}^{1/q+\varepsilon -1}_{q,\sigma}(\Omega) \to  {_\parallel H}^{1/q-\varepsilon -2}_{q,\sigma}(\Omega),$$
hence $B^{\sf vw}$ is a lower order perturbation of $A_{-1}$. Below we set $A^{\sf vw}=A_{-1}-B^{\sf vw}$, 
and observe that $A^{\sf vw}$ also admits a bounded $\cH^\infty$-calculus with $\cH^\infty$-angle 0.

The bilinearity $F^{\sf vw}(w)=G^{\sf vw}(w,w)$ can be estimated as follows
$$|\langle G^{\sf vw}(w_1,w_2)|\phi\rangle|= |(u_1\otimes u_2|\nabla{\rm rot}\, \phi)_\Omega|\leq |u_1|_{L_{2q}(\Omega)}|u_2|_{L_{2q}(\Omega)}|\phi|_{H^2_{q^\prime}(\Omega)},$$
hence $G^{\sf vw}:X_{\beta^{\sf vw}}^{\sf vw}\times X_{\beta^{\sf vw}}^{\sf vw}\to X_0^{\sf vw}$ is bounded, with
$$\beta^{\sf vw}= 1/2+3/4q<1\quad \mbox{and}\quad  \mu_c^{\sf vw}-1/p = 2\beta^{\sf vw}-1= 3/2q.$$
 So here we require $2/p+3/q\leq2$, to have $\mu_c^{\sf vw}\leq1$.
Now we are in position to apply the results from Section 2 to this very weak setting of the vorticity equation, useful to cover the range $q>3/2$, with  critical space $X_{\gamma,\mu_c^{\sf vw}}^{\sf vw} = {_\parallel B}^{3/q-2}_{qp,\sigma}(\Omega)$.

\noindent
\begin{theorem} \label{thm:vwvort}
Let $q\in (3/2,\infty)$, $p\in (1,\infty)$ such that $2/p+3/q\leq2$.

Then, for each  $w_0\in {_\parallel B}^{3/q-2}_{qp,\sigma}(\Omega)$, the vorticity equation \eqref{vorteq}
admits a unique very weak solution
$$ w \in H^1_{p,\mu^{\sf vw}_c}((0,a); {_\parallel H}^{-2}_{q,\sigma}(\Omega)))\cap L_{p,\mu^{\sf vw}_c}((0,a); L_q(\Omega)),$$
for some $a>0$, with critical weight $\mu^{\sf vw}_c = 1/p+ 3/2q$. The solution exists on a maximal interval $(0,t_+(w_0))$ and depends continuously on $w_0$.
In addition,
$$ w \in C([0,t_+); {_\parallel B}^{3/q-2}_{qp,\sigma}(\Omega))
\cap C((0,t_+);{_\parallel B}^{2(1-1/p)-2}_{qp,\sigma}(\Omega)),
$$
i.e.\ the solutions regularize instantly if $2/p +3/q<2$.
\end{theorem}

\medskip

\noindent
\subsection{ Weak Formulation.}
In the weak formulation of the vorticity equation we choose
$X_0^{\sf w}= X_{-1/2}= {_\parallel H}^{-1}_{q,\sigma}(\Omega)$.
 Then for $\phi \in {_\parallel H}^1_{q^\prime,\sigma}(\Omega)$ we obtain with an integration by parts
\begin{align*}
0&= (\partial_t w -\Delta w +{\rm rot}\,(u\cdot\nabla u) |\phi)_\Omega\\
&= (\partial_t w|\phi)_\Omega +(\nabla w|\nabla \phi)_\Omega - (u\cdot\nabla u|{\rm rot}\, \phi)_\Omega
 -(\partial_\nu w|\phi)_\Sigma \\
 &= \langle\partial_t w + A^{\sf w} w|\phi\rangle - \langle F^{\sf w}(w)|\phi\rangle,
\end{align*}
with
$$ \langle F^{\sf w}(w)|\phi\rangle= (u\cdot\nabla u|{\rm rot}\, \phi)_\Omega,\quad \langle A^{\sf w} w|\phi\rangle 
=(\nabla w|\nabla \phi)_\Omega + 
({\rm div}_\Sigma w_\parallel -\kappa_\Sigma w_\nu |\phi_\nu )_\Sigma,
$$ 
and we keep the boundary condition
$$ w_\parallel = \nu\times{\sf B}_\Sigma u \quad \mbox{on } \Sigma.$$
This means
$$X_1^{\sf w}={\sf D}(A^{\sf w})=\{ w\in H^1_q(\Omega)^3:\, {\rm div}\, w=0 \;\mbox{in}\; \Omega, \; w_\parallel = \nu\times {\sf B}_\Sigma u \; \mbox{on} \; \Sigma\}.$$
 Then  the weak formulation of the vorticity equation reads
\begin{equation}\label{vorteq-w}
\partial_t w + A^{\sf w} w = F^{\sf w}(w),\; t>0,\quad w(0)=w_0.
\end{equation}
The operator $A^{\sf w}$  generates its own scale, which differs from that of $A_0$ through the boundary 
condition $w_\parallel =\nu\times B_\Sigma u$. By definition of $A^{\sf w}$ and an integration by parts it follows that $A^{\sf w}=A^{\sf vw}_{1/2}$. 
In particular, $A^{\sf w}$ admits a bounded $\cH^\infty$-calculus with angle 0 as well.

Next, we estimate the bilinearity $F^{\sf w}(w) = G^{\sf w}(w,w)$ as follows.
$$|\langle G^{\sf w}(w_1,w_2)|\phi\rangle|= |(u_1\cdot \nabla u_2|{\rm rot}\, \phi)_\Omega|\leq |u_1|_{L_{qr^\prime}(\Omega)}|u_2|_{H^1_{qr}(\Omega)}|\phi|_{H^1_{q^\prime}(\Omega)},$$
where we choose $r>1$ in such a way that the Sobolev indices of $L_{qr^\prime}$ and $H^1_{qr}$ are equal, which means
$$ 1-3/qr = -3/qr^\prime = -3/q +3/qr,\quad \mbox{i.e.} \quad 3/qr = (1+3/q)/2.$$
This is feasible if $q<3$. Then we have with
$X_{\beta^{\sf w}}^{\sf w} = {_\parallel H}^{2\beta^{\sf w}-1}_{q,\sigma}(\Omega)$
$$ G^{\sf w} :  X_{\beta^{\sf w}}^{\sf w} \times X_{\beta^{\sf w}}^{\sf w} \to X_0^{\sf w} \quad \mbox{bounded},$$
provided
$$\beta^{\sf w} = (1+3/q)/4, \quad \mu^{\sf w}_c -1/p = 2\beta^{\sf w} -1 = (3/q-1)/2.$$
Obviously, $\beta^{\sf w}<1$ and for $\mu_c^{\sf w}\leq1$ we require $ 2/p+ 3/q \leq3$.
As a consequence, the results of Section 2 apply to the vorticity equation in the weak setting for $q<3$, with critical space $ X^{\sf w}_{\gamma,\mu_c^{\sf w}} = {_\parallel B}_{qp,\sigma}^{3/q-2}(\Omega)$, the same spaces as for the very weak formulation in case $3/2<q< 3$. We observe that the Sobolev indices of these critical spaces equal $-2$, i.e.\ it is independent of $q$.

\noindent
\begin{theorem} \label{thm:wvort}
Let $q\in (1,3)$, $p\in (1,\infty)$ such that $2/p + 3/q\leq3$.

Then, for each  $w_0\in {_\parallel B}^{3/q-2}_{qp,\sigma}(\Omega)$, the vorticity equation \eqref{vorteq}
admits a unique weak solution
$$ w \in H^1_{p,\mu^{\sf w}_c}((0,a); {_\parallel H}^{-1}_{q,\sigma}(\Omega))\cap L_{p,\mu^{\sf w}_c}((0,a);  H^1_q(\Omega)),$$
for some $a>0$, with critical weight $\mu^{\sf w}_c = 1/p+ 3/2q -1/2$. 
The solution exists on a maximal interval $(0,t_+(w_0))$ and depends continuously on $w_0$.
In addition,
$$ w \in C([0,t_+); {_\parallel B}^{3/q-2}_{qp,\sigma}(\Omega))\cap C((0,t_+); B^{2(1-1/p)-1}_{qp}(\Omega)^3),
$$
i.e.\ the solutions regularize instantly if $2/p +3/q<3$.
\end{theorem}

\subsection{Conditional Global Existence}
Next we employ the abstract Serrin condition to characterize global existence. For this purpose we only need to compute the spaces $X_{\mu_c^{\sf w}}^{\sf w}$ and $X_{\mu_c^{\sf vw}}^{\sf vw}$. We have
$$ X_{\mu_c^{\sf w}}^{\sf w} = {_\parallel H}^{2\mu_c^{\sf w}-1}_{q,\sigma}(\Omega) = {_\parallel H}^{2/p+3/q-2}_{q,\sigma}(\Omega)={_\parallel H}^{2\mu_c^{\sf vw}-2}_{q,\sigma}(\Omega) = X_{\mu_c^{\sf vw}}^{\sf vw},$$
a surprise? This yields with Theorem \ref{thm5} the following result.

\begin{theorem}  
Let $p\in (1,\infty)$, $q\in (1,\infty)$ such that $s:=2/p+3/q\leq 2$, and $s\leq 3$ in case $q<3$.
 Assume {$w_0\in {_\parallel B}^{3/q-2}_{qp,\sigma}(\Omega)$, and let $w$ denote the unique weak or very weak solution of \eqref{vorteq} 
 according to Theorems \ref{thm:wvort} or \ref{thm:vwvort}, with maximal interval of existence $[0,t_+)$. Then
 \begin{enumerate}
 \item[{\bf (i)}]  $w\in L_p((0,a);{_\parallel H}^{s-2}_{q,\sigma}(\Omega))$, for each $a<t_+$.
 \vspace{1mm}
 \item[{\bf (ii)}] If $t_+<\infty$ then $ w\not\in L_p((0,t_+);{_\parallel H}^{s-2}_{q,\sigma}(\Omega))$.
\end{enumerate}
In particular, the solution exists globally if $w\in L_p((0,a);{_\parallel H}^{s-2}_{q,\sigma}(\Omega))$} for any finite 
number $a$ such that $a\leq t_+$.
\end{theorem}

\noindent
We emphasize the case $s=2$, i.e.\ $2/p+3/q=2$. Then ${_\parallel H}^{s-2}_{q,\sigma}(\Omega)={_\parallel L}_{q,\sigma}(\Omega)$. So we have e.g.\ global existence if $w$ stays bounded in $L_2(J;L_3(\Omega)^3)$ or in $L_4(J;L_2(\Omega)^3)$.

\subsection{Small Data}

In case $p>2, q\geq 3$ we may continue the very weak solution instantly to a weak solution. In fact, for $q\geq 3$ we have the estimate
\begin{equation*}
|G^{\sf w}(w_1,w_2)|_{{_\parallel H}^{-1}_{q,\sigma}} \leq |u_1|_{L_\infty}|u_2|_{H^1_q}
\le C |u_1|_{H^{2\beta}_q} |u_2|_{H^{2\beta}_q} \leq C |w_1|_{H^{2\beta-1}_q} |w_2|_{H^{2\beta-1}_q},
\end{equation*}
for any $2\beta>1$. This shows that any $\mu>1/p$ is admissible, we are in the subcritical case. For sufficiently small $\mu>1/p$ we have 
$$ X_\gamma^{\sf vw}={_\parallel B}^{-2/p}_{qp,\sigma}(\Omega)\hookrightarrow {_\parallel B}^{2(\mu-1/p)-1}_{qp,\sigma}(\Omega)=X^{\sf w}_{\gamma,\mu},$$
hence we obtain the following result.

\begin{theorem}
Let $q\in[3,\infty)$ and $p\in(2,\infty)$.\\
Then, for each  $w_0\in {_\parallel B}^{3/q-2}_{qp,\sigma}(\Omega)$, the vorticity equation \eqref{vorteq}
admits a unique weak solution
$$ u \in H^1_{p,loc}((0,t_+); {_\parallel H}^{-1}_{q,\sigma}(\Omega)))\cap L_{p,loc}((0,t_+); {H}^1_q(\Omega)^3).$$
on a maximal time interval $(0,t_+)$.
\end{theorem}

\noindent
Now we want to consider data which are small in the critical spaces ${_\parallel B}^{3/q-2}_{qp,\sigma}$ where $p,q\in (1,\infty)$, with $2/p+3/q<2$. To apply Corollary
\ref{cor2} in Section 2, we need to study the spectrum of the weak operator $A^{\sf w}$, which consists only of eigenvalues and by elliptic regularity is independent of $q$. So it is sufficient to consider the case $q=2$.

Suppose $\lambda\in \CC$ is an eigenvalue of $A^{\sf w}$ with eigenfunction $w\neq 0$. Then setting $\phi=v$,  with $v$ being the stream function, 
in the definition of the operator $A^{\sf w}$ we obtain
$$\lambda (w|v)_\Omega = \langle A^{\sf w}w|v\rangle,\quad w_\parallel= \nu\times {\sf B}_\Sigma u.$$
Integrating by parts on the left hand side, with $u={\rm rot}\, v$, $w={\rm rot }\,u$, we obtain after some calculations involving the boundary conditions as well as ${\rm div}\, w ={\rm div}\, u={\rm div}\, v =0$ the identity
$$ \lambda|u|_\Omega^2 = 2|D(u)|_\Omega^2 + \alpha |u|_\Sigma^2. $$
In order to verify the assertion on the right hand side, we first use a partial integration to obtain
 $\langle A^{\sf w}w|v\rangle= -(\Delta w|v)_\Omega$ (assuming for the moment that all functions be sufficiently smooth).
 Employing  $-\Delta w={\rm rot}\,{\rm rot}\,w$ (as ${\rm div}\,w=0$) and the relationships between $u,v,w$,
an  integration by parts yields  $(\Delta w|v)_\Omega=(\Delta u|u)_\Omega.$ 
 Yet another integration by parts in conjunction with the Navier boundary conditions and the relation
 $\Delta u = 2\,{\rm div}D(u)$  (as ${\rm div}\,u=0$) implies the assertion.

This shows that the eigenvalues of $A^{\sf w}$ are real and nonnegative. In addition, if $\alpha>0$ then Korn's inequality shows that $0$ is not an eigenvalue. Therefore, the analytic $C_0$-semigroup generated by $A^{\sf w}$ is exponentially stable. Corollary \ref{cor2} then yields the following result.

\begin{theorem}
\label{thm:ex-stable}
Suppose $p,q\in (1,\infty)$ such that $2/p+3/q<2$ and $p\ge2$.\\ 
Then the trivial solution of the vorticity equation \eqref{vorteq} is globally exponentially stable in
$X_{\gamma}^{\sf w} \subset B^{1-2/p}_{qp}(\Omega)^3.$  
Moreover, there is $r_0>0$ such that every very weak solution $w$ with initial value $w_0\in {_\parallel B}^{3/q-2}_{qp,\sigma}(\Omega)$ with norm $|w_0|_{{_\parallel B}^{3/q-2}_{qp,\sigma}}\leq r_0$ exists globally and converges exponentially to zero in the norm of
$B^{1-2/p}_{qp}(\Omega)^3$.
\end{theorem}
\noindent
The Navier boundary conditions considered in  \eqref{NS} were first introduced by Navier in \cite{Na27} and later derived by Maxwell in \cite{Ma79} 
from the kinetic theory of gases. Problem~\eqref{NS}  has been considered by several authors,  
see for instance \cite{PrWi17a} and the references therein.
The construction of a stream function in Subsection 4.2, based on solvability of the elliptic problem~\eqref{LP},  seems to be new. 
Moreover, to the best of our knowledge, the results contained 
in Theorems~\ref{thm:vwvort}-\ref{thm:ex-stable} are new.

 We refer to \cite{Za05} for results concerning the vorticity equations 
 (corresponding to the Naveier-Stokes system with slip boundary conditions) in a cylindrical domain,
 and to  \cite{GiMi89} concerning the vorticity equations in $\R^3$.

\section{Further Applications}
In this section, we apply our theory to a variety of well-known problems.
Applying Theorem \ref{thm1}, as well as Corollaries \ref{cor2} and \ref{cor1},
we obtain results in critical spaces which seem to be new in the case of bounded domains.
We would like to emphasize that our proofs are rather simple, as they do not involve the microscopic structure of Besov spaces.
Corresponding results in the literature seem restricted to the case of $\R^d$, 
where techniques from harmonic analysis are applied.

\subsection{Convection-Diffusion}
Let $\Omega\subset\RR^d$ be a bounded domain of class $C^{4}$
and consider the following non-local convection-diffusion problem.
\begin{equation}
\begin{aligned}
\label{C-D}
\partial_t u -\Delta u &=-{\rm div}(u\nabla w) &&\mbox{in} \;\; \Omega,\\
-\Delta w &= \pm u &&\mbox{in} \;\; \Omega,\\
\partial_\nu u =\partial_\nu w &=0 && \mbox{on} \;\; \partial\Omega,\\
u(0)&=u_0 && \mbox{in} \;\; \Omega.
\end{aligned}
\end{equation}
Here $u$ means a scalar variable, such as a density or a concentration, and $w$ a potential.
Observe that the mean value of $u$ vanishes identically if that of $u_0$ does. 
We assume this throughout. Then $w$ is uniquely defined with mean zero. Without loss of generality, we may assume $\Delta w=u$. If not, we replace $u$ by $-u$. 
Note that this system is (locally) scaling invariant w.r.t\ the scaling
$$\big(u_\lambda,w_\lambda \big)(t,x):=\big(\lambda^2 u,w\big)(\lambda^2t,\lambda x).$$
For $q\in (1,\infty)$, define 
$$X_0:=L_{q,(0)}(\Omega):=\{u\in L_q(\Omega):\int_\Omega u\, dx=0\}$$
and an operator $A:X_1\to X_0$ by $Au=-\Delta u$ with domain 
$$X_1=\{u\in \,H_q^2(\Omega)\cap L_{q,(0)}(\Omega):\partial_\nu u=0\ \text{on}\ \partial\Omega\}.$$
By \cite{DDHPV04}, it holds that $A\in \mathcal{H}^\infty(X_0)$ with angle $\phi_{A}^\infty=0$. 
Moreover, $X_\beta=(X_0,X_1)_\beta={_\nu H_{q,(0)}^{2\beta}(\Omega)}$, $\beta\in [0,1]$, where
\begin{equation*}
{_\nu H_{q,(0)}^{2\beta}(\Omega)}:=L_{q,(0)}(\Omega)\cap
\left\{
\begin{aligned}
& \{u\in H_q^{2\beta}(\Omega):\partial_\nu u=0\ \text{on}\ \partial\Omega\} && \text{for} && 2\beta\in (1+1/q,2],\\
& H_q^{2\beta}(\Omega) &&\text{for} && 2\beta\in [0,1+1/q).
\end{aligned}
\right.
\end{equation*}
For $s\in (1,\infty)$ and $\tau\in [0,2]$, we denote by 
$S: H_{s,(0)}^{\tau}(\Omega)\to H_s^{\tau+2}(\Omega)$ the linear solution map $u\mapsto w$ for the elliptic problem
$$
\begin{aligned}
\Delta w&=u&&\mbox{in} \;\; \Omega,\\
\partial_\nu w&=0&&\mbox{on} \;\; \partial\Omega,
\end{aligned}
$$
which is well-defined thanks to \cite[Theorem 5.5.1]{Tri78} and the fact that $u$ has mean value zero. 
We note on the go that there exists a constant $C>0$ such that for all $u\in H_{s,(0)}^\tau(\Omega)$
$$|Su|_{H_s^{\tau+2}(\Omega)}\le C|u|_{H_s^\tau(\Omega)}.$$
With the operator $S$ at hand, we may reduce \eqref{C-D} to the single equation
\begin{equation}\label{eq:C-D-u}
\partial_t u+Au=F(u),\ t>0,\quad u(0)=u_0,
\end{equation}
where $F(u)=G(u,u)$ and (since $\Delta Sv=v$)
$$G(u,v)= -{\rm div}(u\nabla Sv)= -\nabla u\cdot\nabla Sv - u\Delta Sv= -\nabla u\cdot\nabla Sv - uv$$
is bilinear in $(u,v)\in X_\beta\times X_\beta$. For $u\in X_\beta$ and by H\"older's inequality, we obtain 
$$|F(u)|_{L_q}\le |u|_{H_{qr}^1}|Su|_{H_{qr'}^1}+|u|_{L_{2q}}^2.$$
Choose $\beta\in (0,1)$ such that
$$H_q^{2\beta}(\Omega)\hookrightarrow H_{qr}^1(\Omega),\ H_q^{2\beta+2}(\Omega)\hookrightarrow H_{qr'}^1(\Omega)\quad\text{and}\quad H_{q}^{2\beta}(\Omega)\hookrightarrow L_{2q}(\Omega).$$
The first two embeddings hold if
$$2\beta-\frac{d}{q}=1-\frac{d}{qr}\quad\text{and}\quad 2\beta+1-\frac{d}{q}=-\frac{d}{qr'}.$$
It turns out that this can always be achieved if $q\in (1,d/2)$, hence $d\ge 3$ is necessary. The number $\beta$ can then be computed to the result $\beta=d/4q$.
In particular, the restriction $\beta<1$ is satisfied if $q\in (d/4,d/2)$. Note that for the above value of $\beta$, we have $H_q^{2\beta}(\Omega)\hookrightarrow L_{2q}(\Omega)$.

It is now easy to see that the estimate
$$|F(u)-F(\bar{u})|_{X_0}\le C(|u|_{X_\beta}+|\bar{u}|_{X_\beta})|u-\bar{u}|_{X_\beta}$$
holds for some constant $C>0$ and all $u,\bar{u}\in X_\beta$, hence \textbf{(H2)} is satisfied with $m=\rho=1$ and $\beta_1=\beta$. Thus, the critical weight $\mu_c$ is given by
$$\mu_c=\frac{1}{p}+\frac{d}{2q}-1,$$
which results from \textbf{(H3)}.

 It holds that $\mu_c>1/p$ if $q\in (1,d/2)$ and $\mu_c\le 1$ if $1/p+d/2q\le 2$.
The critical  space reads
$$X_{\gamma,\mu_c}=(X_0,X_1)_{\mu_c-1/p,p}={_\nu}B_{qp,(0)}^{2(\mu_c-1/p)}(\Omega)={_\nu}B_{qp,(0)}^{d/q-2}(\Omega),$$
where
\begin{equation*}
\label{thm:C-D}
{_\nu B}_{qp,(0)}^{r}(\Omega):=L_{q,(0)}(\Omega)\cap 
\left\{
\begin{aligned}
& \{u\in B_{qp}^{r}(\Omega):\partial_\nu u=0\} &&\text{for} && r\in (1+1/q,2],\\
& B_{qp}^{r}(\Omega) &&\text{for} && r\in (0,1+1/q).\\
\end{aligned}
\right.
\end{equation*}
Here we assume that $d\ge 3$, $p\in (1,\infty)$ $q\in (d/4,d/2)$ and $2/p+d/q\le 4$. 

\medskip
Choosing $X_0=H_q^{-1}(\Omega):=(H_{q'}^1(\Omega))'$ as a base space in the weak setting, one obtains
$$\mu_c^{\sf w}=\frac{1}{p}+\frac{d}{2q}-\frac{1}{2}$$
as the critical weight, hence $q<d$, by the condition $\mu_c^{\sf w}>1/p$. Furthermore, $\mu_c^{\sf w}\le 1$ if and only if
$$\frac{1}{p}+\frac{d}{2q}\le \frac{3}{2},$$
hence, in particular, $q>d/3$. This shows that we may consider space dimensions $d\ge 2$ in the weak setting.


\subsection{Electro-Chemistry}\label{subsec:EC}
Let $\Omega\subset\RR^d$ be a bounded domain of class $C^{3}$
and consider the following  problem of Nernst-Planck-Poisson type.
\begin{equation}
\begin{aligned}
\label{NPP}
\partial_t u -\upmu_u\,\Delta u &={\rm div}(u\nabla w) && \mbox{in} \;\; \Omega,\\
\partial_t v -\upmu_v\,\Delta v &=-{\rm div}(v\nabla w) && \mbox{in} \;\; \Omega,\\
\partial_t w-\Delta w &=  u-v && \mbox{in} \;\; \Omega,\\
\partial_\nu u =\partial_\nu v=\partial_\nu w &=0 && \mbox{on} \;\; \partial\Omega,\\
u(0)=u_0,\; v(0)&=v_0 && \mbox{in} \;\; \Omega. 
\end{aligned}
\end{equation}
The variables $u$ and $v$ denote concentrations of oppositely charged ions, and $w$ the induced electrical potential.
Here $ \upmu_u,\upmu_v>0$ are assumed to be constant. In the following, we set $ \upmu_u=\upmu_v=1$. 
Note that this system is scaling invariant w.r.t\ the scaling
$$\big(u_\lambda,v_\lambda, w_\lambda\big)(t,x) :=\big(\lambda^2 u,\lambda^2 v,w\big)(\lambda^2t,\lambda x).$$
Let $B_N=-\Delta$ in $L_q(\Omega)$, $1<q<\infty$, with domain
$${\sf D}(B_N)=\{u\in H_q^2(\Omega):\partial_\nu u=0\ \text{on}\ \Omega\}.$$
It is well-known that for each $\omega>0$, $\omega+B_N\in \mathcal{H}^\infty(L_q(\Omega))$ with angle $\phi_{B_N}^\infty=0$, see e.g.\ \cite{DDHPV04}. The pair $(Y_0,B_0)=(L_q(\Omega),B)$ generates the complex extrapolation-interpolation scale $(Y_\alpha,B_\alpha)$, $\alpha\in\R$, see the Appendix. Consider the operator 
$$B_N^{\sf w}:=B_{-1/2}:H_q^1(\Omega)\to H_q^{-1}(\Omega):=(H_{q'}^1(\Omega))',$$
which has the explicit representation
$$\langle B_N^{\sf w}u|\phi\rangle=(\nabla u|\nabla\phi)_{L_2(\Omega)}$$
for all $(u,\phi)\in H_q^1(\Omega)\times H_{q'}^1(\Omega)$.
Then $B_{N}^{\sf w}\in\mathcal{H}^\infty(H_q^{-1}(\Omega))$ with the same angle as $B_N$.

As a base space for the system variable $z=(u,v,w)$ we take 
\begin{align*}
X_0  &= H^{-1}_q(\Omega)\times  H^{-1}_q(\Omega)\times H^1_q(\Omega)
\end{align*}
and we define $Az:= {\rm diag }\big(B_{N}^{\sf w}, B_{N}^{\sf w},B_{N}|_{H_q^3}\big)z+(0,0,v-u)$, with domain
$$ X_1:={\sf D}(A) = \{ z=(u,v,w)\in H^1_q(\Omega)^2\times H^3_q(\Omega):\; \partial_\nu w=0 \mbox{ on } \partial\Omega\}.$$
By the triangular structure of $A$, it follows readily that $A\in\mathcal{H}^\infty(X_0)$ with $\phi^{\infty}_A=0$.
For the complex interpolation spaces $X_\beta=(X_0,X_1)_\beta$, $\beta\in (0,1)$, we obtain
$$ X_\beta = \{ z=(u,v,w)\in H^{2\beta-1}_q(\Omega)^2\times H^{2\beta+1}_q(\Omega):\; \partial_\nu w=0 \mbox{ on } \partial\Omega\},$$
if $2\beta>1/q$ and
$$X_\beta=H^{2\beta-1}_q(\Omega)^2\times H^{2\beta+1}_q(\Omega)$$
if $2\beta<1/q$,
where $H_q^{-r}(\Omega):=(H_{q'}^r(\Omega))'$ for $r\in [0,1]$.
Define $F=(F_1,F_2,F_3):X_\beta\to X_0$ by
$$ \langle (F_1(z),F_2(z))|(\phi_1,\phi_2)\rangle := \left(-(u\nabla w|\nabla\phi_1)_{L_2},(v\nabla w|\nabla\phi_2)_{L_2}\right),$$
for all $(\phi_1,\phi_2)\in H_{q'}^1(\Omega)^2$ and $F_3(z):=0$. It follows that
$$|F(z)|_{X_0}\le C (|u|_{L_{qr}(\Omega)}+|v|_{L_{qr}(\Omega)})|w|_{H_{qr'}^1(\Omega)},$$
hence there exists a positive number $C$ such that
$$|F(z_1)-F(z_2)|_{X_0}\le C(|z_1|_{X_\beta}+|z_2|_{X_\beta})|z_1-z_2|_{X_\beta},\quad z_1,z_2\in X_\beta,$$
provided $\beta = \frac{1}{4}( 1+\frac{d}{q})$, $q\in (d/3,d)$, which yields the critical space
$$ X_{\gamma,\mu_c} = B^{d/q-2}_{qp}(\Omega)^2\times 
{_\nu}B^{d/q}_{qp}(\Omega)
,$$
where
$${_\nu B}_{qp}^{r}(\Omega):=
\left\{
\begin{aligned}
& \{w\in B_{qp}^{r}(\Omega):\partial_\nu w=0\} &&\text{for} && r\in (1+1/q,3),\\
& B_{qp}^{r}(\Omega) &&\text{for} && r\in (0,1+1/q)\\
\end{aligned}\right.
$$
and $B_{qp}^{-s}(\Omega):=\left(B_{q'p'}^s(\Omega)\right)'$ for $s\in (0,1)$.
Here we assume $ 2/p+d/q\le 3$, to make sure that $\mu_c\le1$. 

\medskip
\noindent
Critical Besov spaces for  system~\eqref{NPP} in the case  $\Omega=\R^d$
and with the parabolic equation  $\partial_t w-\Delta w =  u-v$ is replaced by the corresponding elliptic problem
have been studied in \cite{Zhao17}.

\subsection{Chemotaxis  equations}
Let $\Omega\subset\RR^d$ be a bounded domain of class $C^{3}$.
Then we consider the system
\begin{equation}
\begin{aligned}
\label{CNS}
\partial_t u + u\cdot\nabla u-\upmu_u\, \Delta u +\nabla\pi &=0 && \mbox{in} \;\; \Omega,\\
                                        {\rm div}\,u   &=0 && \mbox{in} \;\; \Omega, \\
\partial_t v + u\cdot\nabla v- \upmu_v\,\Delta v &= -{\rm div}(v\nabla w)  && \mbox{in} \;\; \Omega,\\
\partial_t w + u\cdot\nabla w  -\upmu_w\,\Delta w &=-v && \mbox{in} \;\; \Omega,\\
u=0,\; \partial_\nu v=\partial_\nu w &=0 && \mbox{on} \;\; \partial\Omega,\\
u(0)=u_0,\; v(0)=v_0,\; w(0) &= w_0 && \mbox{in} \;\; \Omega. 
\end{aligned}
\end{equation}
Here $u$ is the velocity field, $w$ denotes the cell density, and $v$ a chemical potential. For simplicity, we choose all constants 
$\upmu_j$, $j\in\{u,w,v\}$, equal to one, as this does not affect the analysis.
Note that the system is (locally) scaling invariant w.r.t. the scaling
$$
\big(u_\lambda ,\pi_\lambda, v_\lambda ,w_\lambda\big)(t,x) :=\big(\lambda u,\lambda^2\pi,\lambda^2 v, w\big)(\lambda^2 t,\lambda x).
$$
Denote by $\mathbb{P}$ the Helmholtz projection in $L_q$ and let $L_{q,\sigma}(\Omega):=\mathbb{P}L_q(\Omega)^d$. 
We choose $X_0=L_{q,\sigma}(\Omega)\times H_q^{-1}(\Omega)\times H_q^1(\Omega)$ as a base space for $z=(u,v,w)$,  where $H_q^{-1}(\Omega):=(H_{q'}^1(\Omega))'$. 
Define a linear operator $A:X_1\to X_0$ by
$$Az={\rm diag }\left(B_S, B_N^{\sf w},B_{N}\right)z-(0,0,v)$$
with domain
$$X_1=\{z=(u,v,w)\in H_{q,\sigma}^2(\Omega)\times H_q^1(\Omega)\times H_q^3(\Omega):u=0,\;\partial_\nu w=0\ \text{on}\ \partial\Omega\}.$$
The operators $B_{N}^{\sf w}$, $B_N$ are defined as in Subsection \ref{subsec:EC} and
$B_Su:=-\mathbb{P}\Delta u.$ 
Furthermore, $A\in \mathcal{H}^\infty(X_0)$ with angle $\phi_A^\infty=0$. The complex interpolation spaces $X_\beta=(X_0,X_1)_\beta$, $\beta\in (0,1)$, then read
$$X_\beta = \{(u,v,w)\in H^{2\beta}_{q,\sigma}(\Omega)\times H^{2\beta-1}_q(\Omega)\times H^{2\beta+1}_q(\Omega):\; u=0,\,\partial_\nu w=0 \mbox{ on } \partial\Omega\}$$
if $2\beta\in (1/q,2]$ and
$$X_\beta =H^{2\beta}_{q,\sigma}(\Omega)\times H^{2\beta-1}_q(\Omega)\times H^{2\beta+1}_q(\Omega)$$
if $2\beta\in [0,1/q)$, where $H_{q,\sigma}^r(\Omega):=H_q^r(\Omega)^d\cap L_{q,\sigma}(\Omega)$ and $H_q^{-s}(\Omega):=(H_{q'}^s(\Omega))'$ for $r\in[0,2]$ and $s\in [0,1]$. 

For $\beta\ge 1/2$, we define $F=(F_1,F_2,F_3):X_\beta\to X_0$ by $F_1(z):=-\mathbb{P}(u\cdot\nabla u)$
$$\langle F_2(z)|\phi\rangle:=(v(u+\nabla w)|\nabla\phi)_{L_2(\Omega)},\quad\phi\in H_{q'}^1(\Omega),$$
and $F_3(z):=-u\cdot \nabla w$. By similar arguments as in Subsection \ref{subsec:EC}, it follows that
there is a constant $C>0$ such that
$$|F(z_1)-F(z_2)|_{X_0}\le C(|z_1|_{X_\beta}+|z_2|_{X_\beta})|z_1-z_2|_{X_\beta},\quad z_1,z_2\in X_\beta,$$
provided $\beta=\frac{1}{4}(1+\frac{d}{q})$ and $q\in (d/3,d)$. This in turn yields the critical space
$$ X_{\gamma,\mu_c} = {_0B}^{d/q-1}_{qp,\sigma}(\Omega)\times B^{d/q-2}_{qp}(\Omega) \times
{_\nu}B^{d/q}_{qp}(\Omega)
,$$
where  $B_{qp}^{-r}(\Omega):=\left(B_{q'p'}^r(\Omega)\right)'$ for $r\in (0,1)$,
\begin{equation*}
{_0 B}_{qp,\sigma}^{r}(\Omega):=L_{q,\sigma}\cap
\left\{
\begin{aligned}
& \{u\in B_{qp}^{r}(\Omega):u=0\} &&\text{for} && r\in (1/q,2),\\
& B_{qp}^{r}(\Omega) &&\text{for} && r\in (0,1/q),\\
\end{aligned}\right.
\end{equation*}
and
\begin{equation*}
{_\nu B}_{qp}^{r}(\Omega):=
\left\{
\begin{aligned}
& \{w\in B_{qp}^{r}(\Omega):\partial_\nu w=0\} &&\text{for} && r\in (1+1/q,3),\\
& B_{qp}^{r}(\Omega) &&\text{for} && r\in (0,1+1/q).\\
\end{aligned}\right.
\end{equation*}

\subsection{Magneto-Hydrodynamics}
In this last subsection, we consider the equations of magneto-hydrodynamics which read
\begin{equation}
\label{MHD1}
\begin{aligned}
\varrho (\partial_t +u \cdot\nabla) u -\upnu\, \Delta u +\nabla \pi &= \frac{1}{\upmu_0} {\rm rot}\, B \times B && \mbox{in}\;\; \Omega,\\
\varrho (\partial_t +u\cdot\nabla) B -\frac{1}{\upmu_0\,\sigma} \Delta B &= B\cdot\nabla u && \mbox{in}\;\; \Omega,\\[2pt]
{\rm div}\, u = {\rm div}\,B &=0 && \mbox{in}\;\; \Omega,\\
u=0,\quad B\cdot\nu =0,\quad \nu\times {\rm rot}\, B &=0  && \mbox{on}\;\; \partial\Omega,\\
u(0)=u_0,\quad B(0)&=B_0 && \mbox{in}\;\; \Omega.
\end{aligned}
\end{equation}
Here $u$ means the velocity field, $\pi$ the pressure, and $B$ the magnetic field. The parameters $\varrho,\upnu,\upmu_0,\sigma>0$ denote physical constants, which we set identical to one in the sequel.
Note that the system is (locally) scaling invariant w.r.t. the scaling
$$  \big(u_\lambda,\pi_\lambda ,B_\lambda\big)(t,x) := \big(\lambda u, \lambda^2 \pi, \lambda B\big)(\lambda^2 t, \lambda x).$$
Next, observe that 
${\rm rot}\, B\times B = B\cdot \nabla B -\frac{1}{2}\nabla |B|^2.$  
As ${\rm div}\, u = {\rm div}\, B=0$ we may rewrite system \eqref{MHD1} in the following way.
\begin{equation}
\label{MHD2}
\begin{aligned}
\partial_t u +{\rm div}\, (u\otimes u) - \Delta u +\nabla \tilde\pi &= {\rm div}\, (B\otimes B) && \mbox{in}\;\; \Omega,\\
\partial_t B + {\rm div}\,(u\otimes B) -\Delta B &= {\rm div}\, (B\otimes u) && \mbox{in}\;\; \Omega,\\
{\rm div}\, u = {\rm div}\,B &=0 && \mbox{in}\;\; \Omega,\\
u=0,\quad B\cdot\nu =0,\quad \nu\times {\rm rot}\, B &=0  && \mbox{on}\;\; \partial\Omega,\\
u(0)=u_0,\quad B(0)&=B_0 && \mbox{in}\;\; \Omega,
\end{aligned}
\end{equation}
with $\tilde \pi =\pi + (1/2)|B|^2$. Let $\mathbb{P}$ denote the Helmholtz projection in $L_{q}(\Omega)^3$ and define $A_0(u,B):=(-\mathbb{P}\Delta u,-\mathbb{P}\Delta B)$ in $X_0:=L_{q,\sigma}(\Omega)^2$ with domain
$$X_1:={\sf D}(A_0):=\{(u,B)\in {H}^2_{q,\sigma}(\Omega)^2:u=0,\;\nu\times{\rm rot}\,B=0\ \text{on}\ \partial\Omega\}.$$
Note that by the properties of the Helmholtz projection, 
$B\cdot \nu=0$  on $\partial\Omega$ for $B\in H^{r}_{q,\sigma}(\Omega)$ with $r\in (1/q,2]$.
The pair $(X_0,A_0)$ generates the complex extrapolation-interpolation scale $(X_\alpha,A_\alpha)$, $\alpha\in\R$, see the Appendix.
In the sequel, we choose the weak setting for $u$ and $B$, i.e.\ $\alpha=-1/2$. This yields $$X_0^{\sf w}:=X_{-1/2}={_0H}_{q,\sigma}^{-1}(\Omega)\times {H}_{q,\sigma}^{-1}(\Omega),$$ where ${_0H}_{q,\sigma}^{-1}(\Omega):=({_0H}_{q,\sigma}^{1}(\Omega))'$, ${H}_{q,\sigma}^{-1}(\Omega):=({H}_{q,\sigma}^{1}(\Omega))'$, 
$${H}_{q,\sigma}^{1}(\Omega):=H_q^1(\Omega)^3\cap L_{q,\sigma}(\Omega)\quad\text{and}\quad{_0H}_{q,\sigma}^{1}(\Omega):=\{u\in {H}_{q,\sigma}^{1}(\Omega):u=0\ \text{on}\ \partial\Omega\}.$$
Denote by $A^{\sf w}$ the operator $A_{-1/2}$ with domain $X_1^{\sf w}:=X_{1/2}={_0H}_{q,\sigma}^{1}(\Omega)\times {H}_{q,\sigma}^{1}(\Omega)$. This way, we may rewrite \eqref{MHD2} as the bilinear evolution equation
$$\partial_t z+A^{\sf w}z=F^{\sf w}(z),\ t>0,\quad z(0)=z_0,$$
in $X_0^{\sf w}$ with $z:=(u,B)$ and $z_0:=(u_0,B_0)$. The nonlinearity $F$ is defined by
$$\langle F^{\sf w}(z)|\phi\rangle:=\left((u\otimes u-B\otimes B|\nabla\phi_1)_{L_2},(u\otimes B-B\otimes u|\nabla\phi_2)_{L_2}\right)$$
for $\phi=(\phi_1,\phi_2)\in {_0H}_{q,\sigma}^{1}(\Omega)\times {H}_{q,\sigma}^{1}(\Omega)$ and $z=(u,B)\in X_\beta$, with
\begin{equation*}
X_\beta={_0\mathbb{H}}_{q,\sigma}^{2\beta-1}(\Omega):=
\left\{
\begin{aligned}
& {_0H}_{q,\sigma}^{2\beta-1}(\Omega)\times {H}_{q,\sigma}^{2\beta-1}(\Omega) &&\text{for} && 2\beta-1\in (1/q,1],\\
& {H}_{q,\sigma}^{2\beta-1}(\Omega)\times {H}_{q,\sigma}^{2\beta-1}(\Omega) &&\text{for} && 2\beta-1\in [0,1/q),
\end{aligned}
\right.
\end{equation*}
and ${_0\mathbb{H}}_{q,\sigma}^{-r}(\Omega):=\left({_0\mathbb{H}}_{q',\sigma}^{r}(\Omega)\right)'$ if $r\in [0,1]$.
By similar arguments as in \cite[Section~5]{PrWi17a} one shows that the critical space is given by
$$X_{\gamma, \mu_c}={_0 B}^{3/q-1}_{q,\sigma}(\Omega)\times {B}^{3/q-1}_{q,\sigma}(\Omega),$$
where we assume  $2/p + 3/q\le 2$ for $p>1$ and $q>3/2$.
We refer to \cite{ZLY17} for corresponding results in the case $\Omega=\R^3$.

\section{Multilinear Nonlinearities}
In this last section, we consider (1.1) with multilinear nonlinearities of the form $F_2(u)=G(u,\ldots,u)$, where
$$ G:\Pi_{k=1}^m X_{\beta_k} \to X_0 $$
is multilinear and bounded, with $\beta_k\in(0,1)$ and $m\geq 2$. 
Although the results derived here are not used in this publication,
they are, nevertheless, relevant for applications.
As before, $X_{\beta_k}=(X_0,X_1)_{\beta_k}$ are complex interpolation spaces.
Here we show how to find the critical weight for this nonlinearity.

For this purpose, we may assume that the sequence $\beta_k$ is non-increasing, i.e.
$$ 1>\beta_1 \geq \beta_2\geq \ldots \geq \beta_m>0,$$
and suppose that $\sum_1^m \beta_k>1$; otherwise we are in the subcritical case as will turn out below. 
Define a sequence $\mu_j$ by means of 
$$ \mu_j = \frac{1}{p} + \frac{1}{j-1}\big(\sum_{k=1}^j \beta_k -1\big),\quad j=2,\ldots,m .$$
For the sake of definiteness, we set $\mu_1=-1$. 
We have \begin{equation*}
\begin{aligned}
\mu_{j+1}\geq\mu_j \quad & \Leftrightarrow \quad\frac{1}{j}\big(\sum_{k=1}^{j+1} \beta_k -1\big) \geq \frac{1}{j-1}\big(\sum_{k=1}^j\beta_k-1\big) \\
&\Leftrightarrow \quad(j-1)\big(\sum_{k=1}^{j+1} \beta_k -1\big) \geq j\big(\sum_{k=1}^{j+1} \beta_k -1\big)-j\beta_{j+1}\\
&\Leftrightarrow \quad \beta_{j+1} \geq \frac{1}{j}\big(\sum_{k=1}^{j+1}\beta_k-1\big) = \mu_{j+1}-\frac{1}{p}.
\end{aligned}
\end{equation*}
We now assume that there is a unique number $l\in\{2,\ldots,m\}$ such that $\mu_l=\max_j \mu_j$. Then the critical weight is given by
$$ \mu_c := \mu_l = \frac{1}{p} + \frac{1}{l-1}\big(\sum_{k=1}^l \beta_k -1\big).$$
We observe that
$\mu_{l}>\mu_{l-1}$ implies $\beta_l>\mu_l-1/p$, 
while $\mu_{l+1}<\mu_{l}$ implies 
$\beta_{l+1}<\mu_{l+1}-1/p<\mu_l-1/p.$
Hence,
$$\beta_j>\mu_c-1/p\;\text{ for $j\le l$,\quad  $\beta_j<\mu_c-1/p\;$ for $j>l$}.$$ 
The assumption $\sum_{k=1}^m \beta_k>1$ yields
$\mu_c-1/p\ge \mu_m-1/p >0$.
On the other side we have $\mu_c-1/p<1$, hence $\mu_c\le 1$ if $p$ is large enough.

We show that with this choice of $\mu_c$, Conditions {\bf (H2)} and  {\bf (H3)} are valid. In fact, the identity
$$ F(u)-F(\bar{u}) = \sum_{j=1}^m G(\bar{u},\ldots,\bar{u}, u-\bar{u}, u,\ldots, u)$$
implies
$$|F(u)-F(\bar{u})|_{X_0} \leq C \sum_{j=1}^m \Pi_{k=1}^{j-1}|\bar{u}|_{X_{\beta_k}}|u-\bar{u}|_{X_{\beta_j}} \Pi_{j+1}^m |u|_{X_{\beta_k}}.$$
Note that 
$X_{\gamma,\mu_c}\hookrightarrow X_{\beta_k}$ for $k>l$ and
$ X_\beta\hookrightarrow X_{\beta_k}\hookrightarrow X_{\gamma,\mu_c}$ for $k\le l$,
where $\beta=\max\beta_k =\beta_1$.
Setting
$$\alpha_k :=\frac{\beta_k-(\mu_c-1/p)}{\beta-(\mu_c-1/p)}$$
we obtain by interpolation
$$ |u|_{X_{\beta_k}}\leq c |u|_{X_{\gamma,\mu_c}}^{1-\alpha_k} |u|_{X_\beta}^{\alpha_k},\quad k\le l.$$
Setting $\alpha_k=0$ for $k>l$, this yields with $\rho_j = \sum_{k\neq j} \alpha_k$ and Young's inequality
\begin{align*}
\Pi_{k=1}^{j-1}|\bar{u}|_{X_{\beta_k}} \Pi_{k=j+1}^m |u|_{X_{\beta_k}}
&\leq C(|u|_{X_{\gamma,\mu_c}},|\bar{u}|_{X_{\gamma,\mu_c}}) |\bar{u}|_{X_\beta}^{\sum_{k=1}^{j-1} \alpha_k}|u|_{X_\beta}^{\sum_{k=j+1}^{m} \alpha_k}\\
&\leq C(|u|_{X_{\gamma,\mu_c}},|\bar{u}|_{X_{\gamma,\mu_c}})\big( |u|_{X_\beta}^{\rho_j}+|\bar{u}|_{X_\beta}^{\rho_j}\big),
\end{align*}
and thus
\begin{equation*}
\Pi_{k=1}^{j-1}|\bar{u}|_{X_{\beta_k}}\Pi_{j+1}^m |u|_{X_{\beta_k}} |u-\bar{u}|_{X_{\beta_j}}
\le C(|u|_{X_{\gamma,\mu_c}},|\bar{u}|_{X_{\gamma,\mu_c}})\big( |u|_{X_\beta}^{\rho_j}+|\bar{u}|_{X_\beta}^{\rho_j}\big)
|u-\bar u|_{X_{\beta_j}}.
\end{equation*}

This shows that Condition {\bf (H2)} holds. 
In order to verify Condition {\bf (H3)} we observe that $\rho_j$ is given by 
$$
\rho_j=\frac{1}{\beta-(\mu_c-1/p)}\sum_{1\le k\le l,k\neq j }(\beta_k-(\mu_c-1/p)).
$$
This yields for $j\le l$
\begin{equation*}
\begin{aligned}
 \rho_j(\beta-(\mu_c-1/p))+ (\beta_j-(\mu_c-1/p) = \big(\sum_{k=1}^l\beta_k\big)-l(\mu_c-1/p)) 
= 1-(\mu_c-1/p),
\end{aligned}
\end{equation*}
showing that $\mu_c$ is critical. 
For $j>l$ we obtain
\begin{equation*}
\begin{aligned}
\rho_j(\beta-(\mu_c-1/p))+ (\beta_j-(\mu_c-1/p))
&=(1-(\mu_c-1/p))+(\beta_j-(\mu_c-1/p)) \\
& < (1-(\mu_c-1/p)).
\end{aligned}
\end{equation*}
Hence $\mu_c$ is subcritical in this case and we may use the estimate
$|u-\bar u|_{X_{\beta_j}}\le c|u-\bar u|_{X_{\gamma,\mu}}$.
Combining, we see that $\mu_c$ defined above is the critical weight for multilinear maps.

\medskip
\goodbreak

\noindent
\begin{remarks}
\label{rem-multi}
Some special cases should be kept in mind.\\
{\bf (i)} $m=2$. Then $l=2$ and $\mu_c -1/p= \beta_1+\beta_2-1$.\\
{\bf (ii)} $m=3$. Then  $\mu_c -1/p= \beta_1+\beta_2-1$ if $\beta_3< \beta_1+\beta_2-1$,  and
 $\mu_c -1/p= (\beta_1+\beta_2+\beta_3-1)/2$ if $\beta_3> \beta_1+\beta_2-1$.\\
 {\bf (iii)} If $\beta_k=\beta$ for all $k$ then $\mu_c = 1/p + (m\beta-1)/(m-1)$.
\end{remarks}

\bigskip
\section{Appendix}
\subsection{Interpolation-extrapolation scale}
Here we collect some basic facts from the theory of Banach scales. Let $X_0$ a reflexive Banach space and $A_0\in\mathcal{BIP}(X_0)$ a linear operator with dense domain $X_1\hookrightarrow X_0$ and $0\in\rho(A_0)$. By \cite[Theorem V.1.5.1]{Ama95} the pair $(X_0,A_0)$ generates an interpolation-extrapolation scale $(X_\alpha,A_\alpha)$, $\alpha\in\R$, with respect to the complex interpolation functor $(\cdot,\cdot)_\theta$, $\theta\in (0,1)$. In particular, for any $\alpha\in\R$, the operator $A_\alpha:X_{\alpha+1}\to X_\alpha$ is a linear isomorphism. If $0\notin \rho(A_0)$, then choose $\omega>0$ such that $0\in\rho(\omega+A_0)$ and replace $A_0$ by $\omega+A_0$.

For $\alpha<\beta$, $\rho(A_\alpha)=\rho(A_\beta)$ and the scale is densely injected, meaning that the embedding $X_\beta\hookrightarrow X_\alpha$ is dense. If $\alpha\ge 0$, $A_\alpha$ is the maximal restriction of $A_0$ to $X_\alpha$ and if $\alpha<0$, then $A_\alpha$ is the closure of $A_0$ in $X_\alpha$, hence
$$A_\alpha u=A_0 u,\quad\text{if}\ u\in X_{1+\max\{\alpha,0\}},\ \alpha\in\R.$$

By \cite[Theorem V.1.5.4]{Ama95} the scale $(X_\alpha,A_\alpha)$, $\alpha\in\R$, is equivalent to the fractional power scale, 
generated by $(X_0,A_0)$. In particular, for $\alpha>0$ it holds that
$$X_\alpha=(D(A^\alpha),|A^\alpha\cdot |)$$
up to equivalent norms and the reiteration property
\begin{equation}
\label{eq:ReitProp}
(X_\alpha,X_\beta)_\theta=X_{(1-\theta)\alpha+\theta\beta},\ \alpha<\beta,\ \theta\in (0,1)
\end{equation}
holds.

Since $X_0$ is reflexive, \cite[Theorem V.1.5.12]{Ama95} yields that $X_\alpha$ is reflexive,
$$(X_\alpha)'=X_{-\alpha}^\sharp\quad\text{and}\quad (A_\alpha)'=A_{-\alpha}^\sharp$$
for any $\alpha\in\R$, where $(X_\alpha^\sharp,A_\alpha^\sharp)$, $\alpha\in\R$, denotes the dual interpolation-extrapolation scale generated by $(X_0^\sharp,A_0^\sharp)$. Here $X_0^\sharp$ is the dual space of $X_0$ and $A_0^\sharp$ denotes the dual operator of $A_0$ in $X_0^\sharp$ with domain $X_1^\sharp$. Furthermore, by \cite[Proposition V.1.5.5]{Ama95}, $A_\alpha\in\mathcal{BIP}(X_\alpha)$ for any $\alpha\in\R$, in particular, the reiteration property \eqref{eq:ReitProp} holds for the dual scale as well. 

Concerning real interpolation of the spaces $X_\alpha$, we note that for all $\alpha,\beta\in [0,1]$ with $(\alpha,\beta)\neq (0,0)$ and $\theta\in (0,1)$, it follows from the reiteration theorem
$$(X_{0},X_{\beta})_{\theta,p}=((X_{-\alpha},X_\beta)_{\frac{\alpha}{\alpha+\beta}},X_\beta)_{\theta,p}=(X_{-\alpha},X_{\beta})_{\frac{\alpha+\theta\beta}{\alpha+\beta},p}$$
and
$$(X_{0},X_{\beta})_{\theta,p}=(X_0,(X_0,X_1)_\beta)_{\theta,p}=(X_{0},X_{1})_{\beta\theta,p},$$
where we made also use of \eqref{eq:ReitProp}. In summary, we obtain
$$(X_{-\alpha},X_\beta)_{\tau,p}=(X_0,X_1)_{\tau(\alpha+\beta)-\alpha,p},$$
provided $\tau>\frac{\alpha}{\alpha+\beta}$ and $\tau<1$. For $\tau\in (0,\frac{\alpha}{\alpha+\beta})$ we make use of duality properties to derive
$$(X_{-\alpha},X_\beta)_{\tau,p}=\left((X_{-\beta}^\sharp,X_{\alpha}^\sharp)
_{1-\tau,p'}\right)'=\left((X_{0}^\sharp,X_{1}^\sharp)
_{\alpha-\tau(\alpha+\beta),p'}\right)',$$
provided $1-\tau>\frac{\beta}{\alpha+\beta}$ or equivalently $\tau<\frac{\alpha}{\alpha+\beta}$.

In particular, if $\alpha=s$ and $\beta=1-s$ for some $s\in [0,1]$, this yields
\begin{equation}
\label{re-int}
(X_{-s},X_{1-s})_{\tau,p}=
\left\{
\begin{aligned}
& (X_0,X_1)_{\tau-s,p} &&\text{for} && \tau\in (s,1)\\
& \left((X_0^\sharp,X_1^\sharp)_{s-\tau,p'}\right)' && \text{for}  &&\tau\in (0,s).
\end{aligned}
\right. 
\end{equation}

\subsection{An interpolation result}
The following interpolation result, which seems to be new, 
was used in the proof of Theorem~\ref{thm5}. 
For the function spaces $\cF_j$ appearing in the next proposition, the reader
should think of $L_{p,\mu}$, $H^s_{p,\mu}$.
\begin{proposition} 
\label{interpol}
Suppose $X_1$ is densely embedded in $X_0$, $A:X_1\to X_0$ is bounded and $A\in \mathcal{BIP}(X_0)$.
Let $\cF_j$, $j=0,1,$ be complete function spaces over an interval $J=(0,a)$ 
and let $\theta\in (0,1)$. Then
\begin{equation*}
(\cF_0(J;X_{\beta_0}),\cF_1(J;X_{\beta_1}))_\theta \cong \cF_\theta(J;X_\beta),
\quad \beta = (1-\theta)\beta_0 + \theta \beta_1,
\end{equation*}
where $(\cdot,\cdot)_\theta$ means complex interpolation,
$\cF_\theta =(\cF_0,\cF_1)_\theta$, and
$X_\alpha=(X_0,X_1)_\alpha$ for $\alpha\in (0,1).$
\end{proposition}
\begin{proof}
As $A$ has bounded imaginary powers, we know that
$X_\alpha ={\sf D}(A^\alpha)$, $\alpha\in (0,1).$
We may assume w.l.o.g that $A$ is invertible.

\medskip
\noindent
{\bf (i)} Let $x\in (\cF_0(J;X_{\beta_0}),\cF_1(J;X_{\beta_1}))_\theta$ be given.
By definition of the complex interpolation method, there exists a bounded and continuous function
$$h:\bar S\to \cF_0(J;X_{\beta_0})+\cF_1(J;X_{\beta_1}),$$
where $S:=[0<{\rm Re}\,z<1]$,
such that $h$ is holomorphic on $S$,  
$$h(i\cdot)\in C_0(\R; \cF_0(J;X_{\beta_0})),\quad 
h(1+i\cdot)\in C_0(\R; \cF_1(J;X_{\beta_1})),\quad \text{and}\;\; x=h(\theta).$$
Here, $C_0$ denotes the space of continuous functions vanishing at infinity.
The norm of $x$ in $(\cF_0(J;X_{\beta_0}),\cF_1(J;X_{\beta_1}))_\theta$ is given by the infimum 
of
$$
|h(i\cdot)|_{L_\infty(\R; \cF_0(J;X_{\beta_0}))} + |h(1+i\cdot)|_{L_\infty(\R; \cF_1(J;X_{\beta_1}))},
$$
taken over all such functions $h$ with $h(\theta)=x$.
Let $g(z):=e^{z^2-\theta^2} A^{\beta_0 + (\beta_1-\beta_0)z}h(z)$ for $z\in\bar S$.
Using the fact that $A\in\mathcal{BIP}(X_0)$ one shows that
\begin{equation*}
g(i\cdot)\in C_0(\R;\cF_0(J;X_0)),\quad g(1+i\cdot)\in C_0(\R;\cF_1(J;X_0)).
\end{equation*}
This implies $h(\theta)=A^\beta x\in \cF_\theta(J;X_0)$
by definition of the complex interpolation method
and we can now conclude that $x\in \cF_\theta(J;X_\beta)$.
The argument also shows that 
$(\cF_0(J;X_{\beta_0}),\cF_1(J;X_{\beta_1}))_\theta$ is embedded in 
$\cF_\theta(J;X_\beta)$.

\medskip
\noindent
{\bf (ii)} 
Suppose $x\in \cF_\theta(J;X_\beta)$. Then there exists a bounded and continuous function
$$g:\bar S\to \cF_0(J;X_{\beta})+\cF_1(J;X_{\beta})$$ 
such that $h$ is holomorphic on $S$, and
$$g(i\cdot)\in C_0(\R; \cF_0(J;X_{\beta})),\quad 
g(1+i\cdot)\in C_0(\R; \cF_1(J;X_{\beta})),\quad \text{and}\;\; x=g(\theta).$$
Let 
$h(z)=e^{z^2-\theta^2} A^{\beta-\beta_0 - (\beta_1-\beta_0)z}g(z)$ for $z\in\bar S$.
Using once more the property that $A$ has bounded imaginary powers one shows that
\begin{equation*}
h(i\cdot)\in C_0(\R;\cF_0(J,X_{\beta_0})),
\quad h(1+i\cdot)\in C_0(\R;\cF_1(J,X_{\beta_1})).
\end{equation*}
Noting that $h(\theta)=x$ we can now conclude that
$x\in (\cF_0(J;X_{\beta_0}),\cF_1(J;X_{\beta_1}))_\theta$,
with continuous embedding.
\end{proof}


\end{document}